\documentclass[12pt,leqno,twoside]{amsart}
\usepackage{amssymb,amsmath,amsthm,soul,color}
\usepackage{t1enc}
\usepackage[cp1250]{inputenc}
\usepackage{a4,indentfirst,latexsym}
\usepackage{graphics}
\usepackage{mathrsfs}
\usepackage{cite,enumitem,graphicx}
\usepackage[colorlinks=true,urlcolor=blue,
citecolor=red,linkcolor=blue,linktocpage,pdfpagelabels,
bookmarksnumbered,bookmarksopen]{hyperref}
\usepackage[english]{babel}
\usepackage[left=2.61cm,right=2.61cm,top=2.72cm,bottom=2.72cm]{geometry}
\usepackage[metapost]{mfpic}
\usepackage[normalem]{ulem}

\usepackage{tikz-cd}

\newtheorem{Th}{Theorem}[section]
\newtheorem{Prop}[Th]{Proposition}
\newtheorem{Lem}[Th]{Lemma}
\newtheorem{Cor}[Th]{Corollary}

\newtheorem{Rem}[Th]{Remark}

\newenvironment{altproof}[1]
{\noindent
	{\em Proof of {#1}}.}
{\nopagebreak\mbox{}\hfill $\Box$\par\addvspace{0.5cm}}

\newcommand{\vp}{\varphi}

\newcommand{\eps}{\varepsilon}

\def\id{\mathrm{id}}

\def\Z{\mathbb{Z}}

\def\N{\mathbb{N}}
\def\R{\mathbb{R}}

\def\cl{\mathrm{cl\,}}

\def\P{\mathcal P} 
\def\U{\mathcal{U}} 

\newcommand{\cC}{{\mathcal C}}

\newcommand{\cG}{{\mathcal G}}

\newcommand{\cK}{{\mathcal K}}

\newcommand{\cM}{{\mathcal M}}

\newcommand{\cO}{{\mathcal O}}
\newcommand{\cP}{{\mathcal P}}

\newcommand{\cS}{{\mathcal S}}

\newcommand{\cU}{{\mathcal U}}

\newcommand{\Om}{\Omega}

\newcommand{\weakto}{\rightharpoonup}

\def\id{\mathrm{id}}

\newcommand{\tu}{\widetilde{u}}
\newcommand{\tv}{\widetilde{v}}

\numberwithin{equation}{section}

\begin{document}
	\title{Nonradial solutions of nonlinear scalar field equations}
	\author[J. Mederski]{Jaros\l aw Mederski}
	\address[J. Mederski]{\newline\indent 
		Institute of Mathematics,
		\newline\indent
		Polish Academy of Sciences,
		\newline\indent 
		ul. \'Sniadeckich 8, 00-956
		Warszawa, Poland
		\newline\indent
		and
		\newline\indent
		Faculty of Mathematics and Computer Science,
		\newline\indent 
		Nicolaus Copernicus University,
		\newline\indent
		ul. Chopina 12/18, 87-100 Toru\'n, Poland}
	\email{\href{mailto:jmederski@impan.pl}{jmederski@impan.pl}}
	\maketitle
	
	\pagestyle{myheadings} \markboth{\underline{J. Mederski}}{
		\underline{Nonradial solutions of nonlinear scalar field equations}}

\begin{abstract}
We prove new results concerning the nonlinear scalar field equation
\begin{equation*}
\left\{
\begin{array}{ll}
-\Delta u  = g(u)&\quad \hbox{in }\mathbb{R}^N,\; N\geq 3,\\
u\in H^1(\mathbb{R}^N)&
\end{array}
\right.
\end{equation*}
with a nonlinearity $g$ satisfying the general assumptions due to Berestycki and Lions. In particular, we find at least one nonradial solution for any $N\geq 4$ minimizing the energy functional on the Pohozaev constraint in a subspace of $H^1(\mathbb{R}^N)$ consisting of nonradial functions.  If in addition $N\neq 5$, then there are infinitely many nonradial solutions. These solutions are sign-changing. The results give a positive answer to a question posed by Berestycki and Lions in \cite{BerLions,BerLionsII}. 
Moreover, we build a critical point theory on a topological manifold, which enables us to solve the above equation as well as to treat new elliptic problems.
\end{abstract}



\section*{Introduction}
\setcounter{section}{1}

We investigate the nonlinear scalar field equation
\begin{equation}\label{eq}
\left\{
\begin{array}{ll}
-\Delta u  = g(u)&\quad \hbox{in }\R^N,\; N\geq 3,\\
u\in H^1(\R^N)&
\end{array}
\right.
\end{equation}
under the following general assumptions introduced by Berestycki and Lions  in their fundamental papers \cite{BerLions,BerLionsII}:
\begin{itemize}
	\item[(g0)] $g:\R \to \R$ is continuous and odd,
	\item[(g1)] $-\infty<\liminf_{s\to 0}g(s)/s\leq\limsup_{s\to 0}g(s)/s=-m<0$,
	\item[(g2)] $\limsup_{s\to \infty}g(s)/s^{2^*-1}=0$, where $2^*=\frac{2N}{N-2}$,
	\item[(g3)] There exists  $\xi_0>0$ such that $G(\xi_0)>0$, where 
	$$G(s)=\int_0^s g(t)\, dt\quad\hbox{for }s\in\R.$$
\end{itemize}

The problem appears e.g. in nonlinear optics or in the the study of Bose–Einstein condensates \cite{Gammal,Buryak} and \eqref{eq} describes the propagation of solitons which are special nontrivial solitary wave solutions $\Phi(x,t)=u(x)e^{-i\omega t}$ of the time-dependent Schr\"odinger equation. The general conditions (g0)--(g3) can model a wide range of nonlinear phenomena, e.g. the Kerr effect, the dual power law nonlinearity or the saturation effect arising in nonlinear optics. The parameter $m$ is usually interpreted as a {\em mass} \cite{BerLions,BerLionsII}.\\
\indent
Recall that the existence of a least energy solution $u\in H^1(\R^N)$, which  is positive, spheri\-cally symmetric (radial) and decreasing in $r=|x|$ is established in \cite{BerLions} and the existence of infinitely many radial solutions but not necessarily positive are provided in \cite{BerLionsII}.
Moreover Jeanjean and Tanaka \cite{JeanjeanTanaka} showed that $J(u)=\inf_{\cM} J$, where $\cM$ stands for the {\em Pohozaev manifold} defined as follows
\begin{equation}\label{def:Poh}
\cM:=\Big\{u\in H^1(\R^N)\setminus\{0\}: \int_{\R^N}|\nabla u|^2\,dx=2^*\int_{\R^N}G(u)\, dx\Big\}.
\end{equation}
and
$J:H^1(\R^N)\to\R$ is is the energy functional  given by
\begin{equation}\label{eq:action}
J(u)=\frac12\int_{\R^N} |\nabla u|^2\,dx - \int_{\R^N} G(u)\, dx.
\end{equation}
Recall that $\cM$ contains all nontrivial critical points \cite{BerLions}.\\
\indent Firstly, the aim of this paper is to answer to the problem \cite{BerLionsII}[Section 10.8] concerning the existence and multiplicity of nonradial solutions to \eqref{eq} for dimensions $N\geq 4$ under the almost optimal assumptions (g0)--(g3). The problem in $N=3$ remains open. Secondly, we present a new variational approach based on a critical point theory built on the Pohozaev manifold. Since $g$ is only continuous, this manifold need not be of class $\cC^{1,1}$ and it seems to be difficult to use a standard  critical point theory directly on this constraint. Without the radial symmetry one has to deal with the issues of lack of compactness and we present a concentration-compactness approach in the spirit of Lions \cite{Lions82,Lions84} together with profile decompositions in the spirit of G\'erard \cite{Gerard} and Nawa \cite{Nawa} adapted to a general nonlinearity satisfying (g0)--(g3); see Theorem \ref{ThGerard}. Using these techniques we provide a new proof of the following result. 

\begin{Th}[\hspace{-0.1mm}\cite{BerLions,JeanjeanTanaka}]\label{ThMain1}
	There is a solution  $u\in \cM$ to \eqref{eq} such that 
	$J(u)=\inf_{\cM} J>0$.
\end{Th}

Moreover we find nonradial solutions to  \eqref{eq} provided that $N\geq 4$. Indeed, let us fix $\tau\in\cO(N)$ such that $\tau(x_1,x_2,x_3)=(x_2,x_1,x_3)$ for $x_1,x_2\in\R^M$ and $x_3\in\R^{N-2M}$, where $x=(x_1,x_2,x_3)\in\R^N=\R^M\times\R^M\times \R^{N-2M}$ and $2\leq M\leq N/2$.
We define
\begin{equation}\label{eq:DefOfX}
X_{\tau}:=\big\{u\in H^1(\R^N): u(x)=-u(\tau x)\;\hbox{ for all }x\in\R^N\big\}.
\end{equation}
Clearly, if $u\in X_\tau$ is radial, i.e. $u(x)=u(\rho x)$ for any $\rho \in\cO(N)$, then $u=0$. Hence $X_\tau$ does not contain nontrivial radial functions. Then $\cO_1:=\cO(M)\times \cO(M)\times\id \subset \cO(N)$ acts isometrically on $H^1(\R^N)$ and let $H^1_{\cO_1}(\R^N)$ denote the subspace of invariant functions with respect to $\cO_1$.\\
\indent Our first main result reads as follows.
\begin{Th}\label{ThMain2} Fix $N\geq 4$ and $2\leq M\leq N/2$. Then there is a solution  $u\in \cM\cap X_\tau\cap H^1_{\cO_1}(\R^N)$ to \eqref{eq} such that 
	\begin{equation}\label{eq:thmain2}
	J(u)=\inf_{\cM\cap X_\tau\cap H^1_{\cO_1}(\R^N)}J\geq 2\inf_{\cM} J.
	\end{equation}
\end{Th}

\indent If in addition $N\neq 5$, then we may assume that $N-2M\neq 1$ and let us consider $\cO_2:=\cO(M)\times \cO(M)\times\cO(N-2M)\subset \cO(N)$ acting isometrically on $H^1(\R^N)$ with the subspace of invariant function denoted by $H^1_{\cO_2}(\R^N)$.
\begin{Th}\label{ThMain3}
Fix $N\geq 4$, $2\leq M\leq N/2$ such that $N-2M\neq 1$.
Then the following statements hold.\\
(a) There is a solution  $u\in \cM\cap X_\tau\cap H^1_{\cO_2}(\R^N)$ to \eqref{eq} such that 
\begin{equation}\label{eq:thmain3}
J(u)=\inf_{\cM\cap X_\tau\cap H^1_{\cO_2}(\R^N)}J\geq \inf_{\cM\cap X_\tau\cap H^1_{\cO_1}(\R^N)}J.
\end{equation}
(b) 
There is an infinite sequence of distinct solutions $(u_n)\subset \cM\cap X_\tau\cap H^1_{\cO_2}(\R^N)$ to \eqref{eq}.
\end{Th}

Note that the associated energy functional $J$ is given by \eqref{eq:action}
is of class $\cC^1$ and has the mountain pass geometry, i.e. $J$ is positive and bounded away from $0$ on a sphere centred at the origin and with a sufficiently small radius $r>0$ and there is $v$ such that $J(v)<0$ and $\|v\|>r$, see  \cite{JeanjeanTanaka} for details. Our problem is modelled in $\R^N$, so that we have to deal with the lack of compactness of Palais-Smale sequences, i.e. sequences $(u_n)$ such that $J(u_n)\to c>0$ and $J'(u_n)\to 0$ as $n\to\infty$. In the classi\-cal approach \cite{BerLions,BerLionsII} the compactness properties can be obtained by considering only radial functions $H^1_{\cO(N)}(\R^N)$ in the spirit of Strauss \cite{Strauss} due to $\cO(N)$-invariance of $J$. In a nonradial case, however, for instance in $H^1(\R^N)$, $X_\tau\cap H^1_{\cO_1}(\R^N)$ or in $X_\tau\cap H^1_{\cO_2}(\R^N)$, the crucial Radial Lemma \cite{BerLions}[Lemma A.II]  is no longer available and an application of the compactness lemma of Strauss \cite{BerLions}[Lemma A.I] is impossible. As usual, one needs to analyse the lack of compactness of Palais-Smale sequences by means of  a concentration-compactness argument of Lions \cite{Lions84}.   
The main difficulty concerning the concentration-compactness analysis is that, in general, $g(s)$ has not subcritical growth of order $s^{p-1}$ for large $s$ with $2<p<2^*$ and $g$ does not satisfy an Ambrosetti-Rabinowitz-type condition \cite{AR}, or any monotonicity assumption, which guarantee the boundedness of Palais-Smale sequences \cite{AR,SzulkinWeth}.  In the present paper we show how to deal with the difficulties with lack of compactness having the general nonlinearity $g$ satisfying (g0)--(g3) and our argument requires a deep analysis of profiles of bounded sequences in $H^1(\R^N)$; see Theorem \ref{ThGerard} below.\\
\indent Beside the lack of compactness difficulties, it is not clear how to treat \eqref{eq} by means of the standard variational methods. Although $J$ has the classical mountain pass geometry \cite{JeanjeanTanaka}, we do not know whether Palais-Smale sequences of $J$ are bounded. To overcome this difficulty in the radial case in \cite{BerLions,BerLionsII}, the authors considered the following constrained problems: the minimization of $u\mapsto \int_{\R^N}|\nabla u|^2\, dx$ on 
$$\Big\{u\in H^1_{\cO(N)}(\R^N): \int_{\R^N}G(u)\, dx=1\Big\}$$ and a critical point theory of the functional $u\mapsto \int_{\R^N} G(u)\,dx$ on 
$$\Big\{u\in H^1_{\cO(N)}(\R^N): \int_{\R^N}|\nabla u|^2\, dx=1\Big\}.$$ Both approaches require the compactness properties and the scaling invariance of the equation \eqref{eq} with application of Lagrange multipliers. Another method in the radial case in \cite{Hirata} is based on the Mountain Pass Theorem for an extended functional in the spirit of Jeanjean \cite{Jeanjean}.  Let us mention that a direct minimization method on the Pohozaev manifold in $H^1_{\cO(N)}(\R^N)$ is due to  Shatah \cite{Shatah}, who studied a nonlinear Klein-Gordon equation with a general nonlinearity. Again, the radial symmetry and the Strauss lemma played an important role in these works.\\
\indent In this paper we provide a new constrained approach which allows to deal with noncompact problems and can be described in an abstract and transparent way for future applications; see Section \ref{sec:criticaltheory} for details. Let us briefly sketch our approach. Recall that if $u\in H^1(\R^N)$ is a critical point of $J$, then
$u\in W^{2,q}_{loc}(\R^N)$ for any $q<\infty$ and $u$ satisfies the {\em Pohozaev identity}, i.e. $M(u)=0$, where
$$M(u):=\int_{\R^N}|\nabla u|^2\,dx-2^*\int_{\R^N}G(u)\, dx.$$
Observe that $M:H^1(\R^N)\to\R$ is of class $\cC^1$ and but, in general, $M'$ is not locally Lipschitz and 
$$\cM=\{u\in H^1(\R^N)\setminus\{0\}: M(u)=0\}$$
need not be of class $\cC^{1,1}$. Hence it seems to be impossible to use any critical point theory based on the deformation lemma involving a Cauchy problem directly on $\cM$, e.g. as in \cite[Section 5]{Willem}. Our crucial observation is that $\cM$ is a topological manifold and there is a homeomorphism $m_\cU:\cU\to\cM$ such that
\begin{equation}\label{PdefofU}
\cU:=\Big\{u\in H^1(\R^N): \int_{\R^N} |\nabla u|^2\, dx=1 \hbox{ and }\int_{\R^N}G(u)\, dx>0\Big\}
\end{equation}
is a manifold of class $\cC^{1,1}$.
Moreover $J\circ m_\cU:\cU\to \R$ is still of class $\cC^1$ and $u\in\cU$ is a critical point of $J\circ m_\cU$ if and only if $m_\cU(u)$ is a critical point of the unconstrained functional $J$. The main difficulty is the fact that it is not clear whether a Palais-Smale sequence $(u_n)\subset \cU$ of $J\circ m_\cU$ can be mapped into a Palais-Smale sequence $m_\cU(u_n)\subset \cM$ of the unconstrained functional $J$. Moreover, we do not know if a nontrivial weak limit point of $(m_\cU(u_n))$ is a critical point of $J$ and stays in $\cM$.\\ 
\indent In order to overcome these obstacles we introduce a new variant of the Palais-Smale condition at level $\beta\in\R$ denoted by $(M)_\beta\; (i)$ (see Section \ref{sec:criticaltheory}), which roughly says that, for
every Palais-Smale sequence $(u_n)\subset \cU$ at level  $\beta$, $(m_\cU(u_n))$
contains a subsequence converging weakly towards a point $u\in H^1(\R^N)$ up to the $\R^N$-translations, which can be projected on a critical point $m_{\cP}(u)\in \cM$. Moreover we may choose a proper sequence of $\R^N$-translations such that
$$J(m_{\cP}(u))\leq \beta=\lim_{n\to\infty} J(m_\cU(u_n)).$$
The selection of the proper translation plays a crucial role and requires the following profile decompositions of bounded sequences in $H^1(\R^N)$ in the spirit of \cite{Gerard,HmidiKeraani,SoliminiDev}.

\begin{Th}\label{ThGerard}
	Suppose that $(u_n)\subset H^{1}(\R^N)$ is bounded.
	Then there are sequences
	$(\tu_i)_{i=0}^\infty\subset H^1(\R^N)$, $(y_n^i)_{i=0}^\infty\subset \R^N$ for any $n\geq 1$, such that $y_n^0=0$,
	$|y_n^i-y_n^j|\rightarrow \infty$ as $n\to\infty$ for $i\neq j$, and passing to a subsequence, the following conditions hold for any $i\geq 0$:
	\begin{eqnarray}
	\nonumber
	&& u_n(\cdot+y_n^i)\weakto \tu_i\; \hbox{ in } H^1(\R^N)\text{ as }n\to\infty,\\
	\label{EqSplit2a}
	&& \lim_{n\to\infty}\int_{\R^N}|\nabla u_n|^2\, dx=\sum_{j=0}^i \int_{\R^N}|\nabla\tu_j|^2\, dx+\lim_{n\to\infty}\int_{\R^N}|\nabla v_n^i|^2\, dx,
	\end{eqnarray}
	where $v_n^i:=u_n-\sum_{j=0}^i\tu_j(\cdot-y_n^j)$ and
	\begin{eqnarray}
	&& \limsup_{n\to\infty}\int_{\R^N}\Psi(u_n)\, dx= \sum_{j=0}^i
	\int_{\R^N}\Psi(\tu_j)\, dx+\limsup_{n\to\infty}\int_{\R^N}\Psi(v_n^i)\, dx	\label{EqSplit3a}
	\end{eqnarray}
for any function $\Psi:\R\to[0,\infty)$ of class $\cC^1$ such that $\Psi'(s)\leq C(|s|+|s|^{2^*-1})$ for any $s\in\R$ and some constant $C>0$.
	Moreover, if in addition $\Psi$ satisfies
		\begin{equation}
		\label{eq:Psi2}\lim_{s\to 0} \frac{\Psi(s)}{|s|^{2}}=\lim_{|s|\to\infty} \frac{\Psi(s)}{|s|^{2^*}}=0,
		\end{equation}
		then
	\begin{equation}\label{EqSplit4a}
	\lim_{i\to\infty}\Big(\limsup_{n\to\infty}\int_{\R^N}\Psi(v_n^i)\, dx\Big)=0.
	\end{equation}
\end{Th}
In particular, taking $\Psi(s)=|s|^p$ with $p=2$ and with $2<p<2^*$ we obtain \cite{HmidiKeraani}[Proposition 2.1]. 
Our argument relies only on new variants of Lions lemma \cite{Willem}[Lemma 1.21]; see Section \ref{sec:Lions} with Lemma \ref{lem:Lions} and variants of Theorem \ref{ThGerard} in $H^1_{\cO_1}(\R^N)$ as well as in $H^1_{\cO_2}(\R^N)$.\\
\indent Having a minimizing sequence of $J\circ m_\cU$, we find a proper translation such that a weak limit point can be projected on a critical point of $J$ in $\cM$ and we prove Theorem \ref{ThMain1}. The same procedure works in the subspace $X_\tau\cap H^1_{\cO_1}(\R^N)\subset H^1(\R^N)$, however we have to ensure that we choose a proper translation along $\R^{N-2M}$-variable and we get Theorem \ref{ThMain2}.\\
\indent In order to get multiplicity of critical points, we show that
$J\circ m_\cU$ satisfies the Palais-Smale condition in  $\cU\cap X_\tau\cap H^1_{\cO_2}(\R^N)$ and in view of the critical point Theorem \ref{Th:CrticMulti} of Section \ref{sec:criticaltheory}, $J$ has infinitely many critical points and we prove Theorem \ref{ThMain3}.\\
\indent Note that the existence and the multiplicity results concerning similar problems to \eqref{eq} in the noncompact case present in the literature require strong growth conditions imposed on the nonlinear term, e.g. $g$ has to be of subcritical growth, i.e. $|g(s)|\leq c(|s|+|s|^{p-1})$ for some constant $c>0$ and $2<p<2^*$, and, in addition, must satisfy an Ambrosetti-Rabinowitz-type condition \cite{AR,CotiZelatiRab},  or a monotonicity-type assumption \cite{SzulkinWeth}; see also  references therein.
If a nonlinear equation like \eqref{eq} exhibits radial symmetry, then the problem of existence of nonradial solutions is particularly challenging and there are only few results in this direction. The first paper \cite{BartschWillem} due to Bartsch and Willem dealt with semili\-near elliptic problems in dimension $N=4$ and $N\geq 6$ under subcritical growth conditions and an Ambrosetti-Rabinowitz-type condition. In fact, from \cite{BartschWillem} we borrowed an idea of the decomposition of $\R^N$ into $\R^M\times\R^M\times\R^{N-2M}$ and  $\cO_2$-action on $H^1(\R^N)$ given in Theorem \ref{ThMain2}. Further analysis of decompositions of $\R^N$ in this spirit has been recently studied in \cite{Marzantowicz} and in the references therein.  Next, Lorca and Ubilla \cite{Lorca} solved the similar problem in dimension $N=5$ by considering  $\cO_1$-action on $H^1(\R^N)$, and recently Musso, Pacard and Wei \cite{Musso} obtained nonradial solutions in any dimension $N\geq 2$; see also \cite{AoWei}. 
In these works, again, strong assumptions needed to be imposed on nonlinear terms, for instance a nondegeneracy condition in \cite{Musso,AoWei}, which  allows to apply a Liapunov-Schmidt-type reduction argument.\\
\indent The paper is organized as follows. In Section 2 we build a critical point theory on a general topological manifold $\cM$ in the setting of abstract assumptions (A1)--(A3). Having our variants of Palais-Smale condition $(M)_\beta\;(i)-(ii)$, in Theorem \ref{Th:CrticMulti} we prove the existence of minimizers on $\cM$ and the multiplicity result. The general theorem can be useful in the study of strongly indefinite problems as well, like \cite{SzulkinWeth,MederskiENZ}, where the classical linking approach due to Benci and Rabinowitz \cite{BenciRabinowitz} does not apply and the classical Palais-Smale condition is not satisfied; see Remark \ref{rem:CrtiticalPointTheory}. Moreover, in a subsequent work \cite{MederskiZeroMass} 
these techniques will be used to obtain nonradial solutions in the zero mass case problem \eqref{eq}, which has been studied in the radial case so far in \cite{BerLions,BerLionsInfZero}. In Section \ref{sec:Lions} we prove three variants of Lions lemma in $H^1(\R^N)$, $H^1_{\cO_1}(\R^N)$ and in $H^1_{\cO_2}(\R^N)$. These allow us to prove the profile decomposition Theorem \ref{ThGerard} and its variant Corollary \ref{CorGerard} in order to analyse Palais-Smale sequences in Proposition \ref{prop:PSanaysis} and in Corollary \ref{cor:PSanalysis}. We complete proofs of Theorems \ref{ThMain1}, \ref{ThMain2} and \ref{ThMain3} in the last Section \ref{sec:proof}.\\
\indent Finally we would like to point out that similarly as in \cite{BerLions} one can assume (g0), (g1), (g3) and that $-\infty\leq \limsup_{s\to\infty} g(s)/s^{2^*-1}\leq 0$ instead of (g2). Then, in view of Theorem \ref{ThMain2} and Theorem \ref{ThMain3} we obtain nonradial solutions as well. Indeed,
we modify $g$ in the following way. If $g(s)\geq 0$ for all $s\geq \xi_0$, then $\tilde{g}=g$. Otherwise we set
$\xi_1:=\inf\{\xi\geq \xi_0: g(\xi)=0\}$, 
\begin{equation*}
\tilde{g}(s)=\left\{
\begin{array}{ll}
g(s)&\quad\hbox{if }0\leq s \leq \xi_1,\\
0 &\quad\hbox{if }s>\xi_1,
\end{array}
\right.
\end{equation*}
and $\tilde{g}(s)=-\tilde{g}(-s)$ for $s<0$. Hence $\tilde{g}$ satisfies (g0)--(g3), and by the strong maximum principle if $u\in H^1(\R^N)$ solves $-\Delta u=\tilde{g}(u)$, then $|u(x)|\leq \xi_1$ and $u$ is a solution to \eqref{eq}.

\section{Critical point theory on a topological manifold}\label{sec:criticaltheory}

In this section we introduce a critical point theory on a topological submanifold $\cM$ of a Banach space endowed with a certain isometric group action. As we shall see we introduce variants of the Palais-Smale condition $(M)_\beta\;(i)$ and $(M)_\beta\;(ii)$ which allow to treat variational problems with the lack of compactness. The results of this section will be applied to \eqref{eq} in Section \ref{sec:proof}.\\
\indent
Let $G$ be an isometric group action on a reflexive Banach space $X$ with norm $\|\cdot\|$ and $J:X\to \R$ is a $\cC^1$-functional. Assume that
\begin{itemize}
	\item[(A1)] $J$ is $G$-invariant, i.e. if $u\in X$ and $g\in G$ then $J(gu)=J(u)$. If $gu=-u$, then $u=0$. Moreover if $g_n\in G$, $u\in X$ and $g_nu\weakto  v$, then $v=gu$ for some $g\in G$ or $v=0$. 
\end{itemize}
Let $\cM\subset X\setminus\{0\}$ be a closed and nonempty subset of $X$ such that
\begin{itemize}
	\item[(A2)] $\cM$ is $G$-invariant and $\inf_{\cM} J>0=J(0)$.
\end{itemize}
Since, in general, $\cM$ has not the $\cC^{1,1}$-structure, it is difficult to build a critical point theory directly on $\cM$ and deal with deformation techniques involving a Cauchy problem, e.g. as in \cite[Section 5]{Willem}. Therefore we
 introduce a manifold
$$\cS=\{u\in Y: \psi(u)=1\}$$
in a closed $G$-invariant subspace $Y\subset X$, where $\psi\in\cC^{1,1}(Y,\R)$ is $G$-invariant and such that $\psi'(u)\neq 0$ for $u\in\cS$. Clearly, from the implicit function theorem,  $\cS$ is a $G$-invariant manifold of class $\cC^{1,1}$ and of codimension $1$ in $Y$ with the following tangent space at $u\in\cS$
$$T_u \cS=\{v\in Y: \psi'(u)(v)=0\}.$$
Moreover we assume:
\begin{itemize}
	\item[(A3)] There are a $G$-invariant open neighbourhood $\cP\subset X\setminus\{0\}$ of $\cM$ and $G$-equivariant map $m_{\cP}:\cP\to \cM$ such that $m_{\cP}(u)=u$ for $u\in\cM$ and the restriction $m_\cU:=m_{\cP}|_{\cU}:\cU\to\cM$ for $\cU:=\cS\cap\cP$  is a homeomorphism. Moreover
	$J\circ m_\cU=J|_{\cM}\circ m_\cU$ is of class $\cC^1$ and $(J\circ m_\cU)(u_n)\to\infty$ as $u_n\to u\in\partial\cU$, $u_n\in\cU$, where the boundary of $\cU$ is taken in $\cS$.
The above sets are described in the following diagram:
\[
\begin{tikzcd}
X\rar[phantom,"\supset" ] \rar[d,phantom,"\cup" ] &X\setminus\{0\} \rar[phantom,"\supset" ] &\cP\arrow[r, "m_\cP"] \rar[d,phantom,"\cup" ]  & \cM   \\
Y \rar[phantom,"\supset" ] & \cS \rar[phantom,"\supset"] & \cU=\cS\cap\cP\arrow[ur, "m_\cU", swap] \arrow[ur, "\approx"]& 
\end{tikzcd}
\]
\end{itemize}

\indent As usual, we say that $(u_n)\subset \cU$ is a {\em $(PS)_\beta$-sequence} of $J\circ m_\cU:\cU\to\R$ 
provided that
$$(J\circ m_\cU)(u_n)\to \beta\hbox{ and }(J\circ m_\cU)'(u_n) \to 0.$$
Let $\cK$ be the set of all critical points of $J\circ m_\cU$, i.e.
$$\cK:=\big\{u\in\cU: (J\circ m_\cU)'(u)(v)=0\hbox{ for any }v\in T_u\cS\big\}.$$
 For $u\in X$, $G\ast u$ denotes the orbit of $u$
$$G\ast u:=\{gu: g\in G\}.$$ 
$G\ast u$ is called a {\em critical orbit} if $u\in\cK$. Clearly $G\ast u\subset \cK$ if $u\in\cK$.
We introduce the following variants of the {\em Palais-Smale condition} at level $\beta\in\R$.

\begin{itemize}
	\item[$(M)_\beta\;(i)$]  For every $(PS)_{\beta}$-sequence $(u_n)\subset \cU$ of $J\circ m_\cU$, there are a sequence $(g_n)\subset G$ and $u\in\cP$ such that
	$g_nm_\cU(u_n) \weakto u$ along a subsequence, $J'(m_{\cP}(u))=0$ and $J(m_{\cP}(u))\leq \beta$.
	\item[$\hspace{1cm}(ii)$]
	If the number of distinct critical orbits is finite,  then there is $m_\beta>0$ such that for every $(u_n)\subset \cU$ such that $(J\circ m_\cU)'(u_n)\to 0$ as $n\to\infty$, $(J\circ m_\cU)(u_n)\leq\beta$ and $\|u_n-u_{n+1}\|<m_\beta$ for $n\geq 1$, there holds
	$\liminf_{n\to\infty}\|u_n-u_{n+1}\|=0$.
\end{itemize}

Note that  $(M)_\beta(i)$ implies that if $(J\circ m_\cU)'(u)=0$, then $J'(m_\cU(u))=0$ for $u\in\cU$.
Indeed, taking a sequence $u_n=u$, observe that $g_n m_\cU(u)\weakto \tu$ along a subsequence, $\tu\in \cP\subset X\setminus\{0\}$ and by (A1), $\tu=gm_\cU(u)$ for some $g\in G$. Then $\tu\in \cM$ and by (A3), $m_{\cP}(\tu)=m_{\cP}(gm_\cU(u))=gm_\cU(u)$ is a critical point of $J\circ m_\cU$, hence by (A1), we conclude $J'(m_\cU(u))=0$.
Therefore critical points of $J\circ m_\cU$ are mapped by $m_\cU$ into nontrivial critical points of the unconstrained functional $J$. Observe that, however, $m_\cU(u_n)$ need not to be a Palais-Smale sequence of the unconstrained functional $J$ if $(u_n)\subset \cU$ is a $(PS)_\beta$-sequence of $J\circ m_\cU$.\\
\indent In what follows, for $A\subset X$ and $r>0$, $B(A,r):=\{u\in X: \|u-v\|<r\hbox{ for some }v\in A\}$. Moreover,  let  $\Phi:=J\circ m_\cU:\cU\to\R$ and
for any $\alpha<\beta$ let us denote 
\begin{eqnarray*}
	\Phi^{\beta}_\alpha&:=&\{u\in\cU: \alpha\leq \Phi(u)\leq \beta\},\\
	\Phi^{\beta}&:=&\{u\in\cU: \Phi(u)\leq \beta\}.
\end{eqnarray*}	

\begin{Lem}\label{lem:M_betacond}
	Suppose that (A1)--(A3), $(M)_\beta\; (ii)$ hold for some $\beta\in\R$ the number of distinct critical orbits is finite. If $(u_n)\subset\cU$ is a $(PS)_\alpha$-sequence for some $\alpha<\beta$ and
	\begin{equation}\label{eq:closetoK}
	u_n\in B(\cK\cap \Phi^\beta,m_\beta),
	\end{equation}
	then passing to a subsequence $g_nu_n\to u$ for some $u\in \cK$ and $g_n\in G$.
\end{Lem}
\begin{proof}
	Let $(u_n)\subset\cU$ be a $(PS)_\alpha$-sequence such that \eqref{eq:closetoK} holds. Passing to a subsequence $\Phi(u_n)\leq \beta$. Then 
	we put $w_{2n-1}:=u_n$ and take any $w_{2n}\in \cK$ such that 
	$$\|w_{2n}-u_n\|<m_\beta$$
	and  $\Phi(w_{2n})\leq \beta$ 
	for any $n\geq 1$. Take $\tilde{\cK}\subset \cK$ such that each orbit $G\ast u$ has a unique representative in $\tilde{\cK}$ for $u\in \cK$, so that $\tilde{\cK}\cap (G\ast u)$ is a singleton. Since $\tilde{\cK}$ is finite, 
	passing to a subsequence we may assume that $g_{2n}w_{2n}=u\in \tilde{\cK}$ for some $g_{2n}\in G$ and $u\in \tilde{\cK}$. Take $g_{2n-1}=g_{2n}$ for $n\geq 1$ and observe that by (A1), $\Phi'$ is $G$-equivariant, hence $\Phi'(g_nw_n)\to 0$, $\Phi(g_nw_n)\leq \beta$ and
	$$\|g_nw_n-g_{n+1}w_{n+1}\|<m_{\beta}$$
	for $n\geq 1$.
	Then, in view of $(M)_{\beta}\; (ii)$ we obtain
	$$\liminf_{n\to\infty}\|g_{2n}u_n-u\|=\liminf_{n\to\infty}\|g_{2n-1}w_{2n-1}-g_{2n}w_{2n}\|=0,$$
	and $g_{2n}u_n\to u\in \cK$ passing to a subsequence.
\end{proof}
Hence, roughly speaking, Lemma \ref{lem:M_betacond} says that if $(M)_\beta\; (ii)$ holds, then provided that $\cK$ has a finite number of distinct orbits, a Palais-Smale sequence of $\Phi=J\circ m_\cU$ which is close enough to $\cK$ will contain  a convergent subsequence up to the $G$-action.\\
\indent In order to deal with the multiplicity of critical points by means of the Krasnoselskii genus \cite{Struwe}, we consider the following assumption:
\begin{itemize}
	\item[(S)] $J$ is even, $m_{\cP}$ is odd, $\cU$, $\cM$ are symmetric, i.e. $\cU=-\cU$, $\cM=-\cM$ and
 for any $k\geq 1$, there exists a continuous and odd map $\tau:S^{k-1}\to\cP$,
where $S^{k-1}$ is the unit sphere in $\R^k$.
\end{itemize}
As we shall see later, the latter condition in (S) guarantees that the Krasnoselskii genus of $\Phi^\beta$ is sufficiently large for large $\beta$.\\
\indent Now our main result of this section reads as follows.

\begin{Th}\label{Th:CrticMulti}
	Suppose that $J:X\to\R$ is of class $\cC^1$ and satisfies (A1)--(A3).\\
	(a) If  $(M)_{\beta}\; (i)$ holds for $\beta=\inf_{\cM} J$, then
	$J$ has a critical point $u\in\cM$ such that $$J(u)=\inf_{\cM}J.$$
	(b) Assume that $(M)_{\beta}\; (i)-(ii)$ hold for every $\beta\geq \inf_{\cM} J$ and (S) is satisfied.
	Then $J$ has infinitely many $G$-distinct critical points in $\cM$, i.e. there is a sequence of critical points $(u_n)\subset\cM$ such that $(G\ast u_n)\cap (G\ast u_m)=\emptyset$ for $n\neq m$.\\
	(c) Assume that $G=\{\id\}$, (S) is satisfied and for every $(PS)_{\beta}$-sequence $(u_n)\subset \cU$ of $J\circ m_\cU$ with $\beta\geq \inf_{\cM} J$, there is $u\in\cU$ such that $J'(m_\cU(u))=0$ and $u_n \to u$ along a subsequence. 
	Then $J$ has infinitely many critical points in $\cM$.
\end{Th}

\begin{proof}
Proof of (a).	 Similarly as in \cite[Lemma 5.14]{Willem} we find
	an odd and  locally Lipschitz pseudo-gradient vector field $v:\cU\setminus \cK\to Y$ such that  $v(u)\in T_u\cS$ and
	\begin{eqnarray}
	\|v(u)\|&<&2\| \Phi'(u)\|,\label{eq:flow2}\\
	\Phi'(u)(v(u)) &>& \|\Phi'(u)\|^2\label{eq:flow3}
	\end{eqnarray}
	for any $u\in \cU\setminus \cK$. The obtained pseudo-gradient vector field allows to prove a variant of deformation lemma \cite[Lemma 5.15]{Willem} in 
$\cU$ and arguing as in  \cite[Theorem 8.5]{Willem}, we find a minimizing sequence $(u_n)\subset \cU$ such that 
	$$\Phi(u_n)\to c:=\inf_{\cU}J\circ m_\cU=\inf_{\cM}J$$
	and $\Phi'(u_n)\to 0$ as $n\to\infty$. In view of $(M)_{\beta}\;(i)$ we find a nontrivial critical point $m_{\cP}(u)\in \cM$ of $J$ such that passing to a subsequence $g_nu_n\weakto u$ for some $g_n\in G$. Since 
	$$c\geq J(m_{\cP}(u))\geq\inf_{\cM}J,$$
	we get $J(m_{\cP}(u))=c$, which completes proof of (a). \\
\indent 	Proof of (b) and (c) is based on the fact that the Lusternik-Schnirelman values
		\begin{equation}\label{eq:LSvalue}
			\beta_k:= \inf\{\beta\in\R:  \gamma(\Phi^{\beta})\geq k\}.
		\end{equation}
	are increasing critical values of $\Phi$ for $k\geq 1$, where $\gamma$ stands for the Krasnoselskii genus for closed and symmetric subsets of $X$, see \cite[Chapter II.5]{Struwe} for the definition and properties of the genus. Observe that (S) implies that for any $k\geq 1$ there is an odd map $\tau:S^{k-1}\to\cP$, hence 
by \cite[Proposition II.5.4 and Proposition II.5.2]{Struwe} we get
$\gamma\Big(m^{-1}\big(m_{\cP}(\tau(S^{k-1}))\big)\Big)\geq 
\gamma\big(\tau(S^{k-1})\big)\geq \gamma\big(S^{k-1}\big)=k$. Therefore
we find $\beta>0$ such that 
	$$\gamma(\Phi^{\beta})\geq \gamma\Big(m^{-1}\big(m_{\cP}(\tau(S^{k-1}))\big)\Big)\geq k,$$
	hence $\beta_k<\infty$.
	Now, if the classical Palais-Smale condition is satisfied as in (c), then one can argue as in  \cite[Theorem 8.10]{Rabinowitz:1986} and show that $(\beta_k)$ is a sequence of increasing critical values and we conclude (c).\\
\indent Proof of (b).	
Observe that we find the unique flow 
	$\eta:\cG\to  \cU\setminus \cK$ for the pseudo-gradient vector field obtained in (a) such that
	\begin{equation*}
	\left\{
	\begin{aligned}
	&\partial_t \eta(t,u)=-v(\eta(t,u))\\
	&\eta(0,u)=u
	\end{aligned}
	\right.
	\end{equation*}
	where $\cG:=\{(t,u)\in [0,\infty)\times (\cU\setminus \cK):\; t<T(u)\}$ and $T(u)$ is the maximal time of the existence of $\eta(\cdot,u)$. 
Suppose that the number of distinct critical orbits is finite. We will show that in fact $(\beta_k)$ is a sequence of increasing critical values and this contradiction will complete the proof.
	Take $\beta\geq c$ and let
	\begin{eqnarray*}
		\cK^\beta&:=&\{u\in \cK:  \Phi(u)=\beta\}.
	\end{eqnarray*}
{\em Claim 1.}	There is $\eps_0>0$ such that 
	\begin{equation}\label{eq:KL}
	\cK\cap \Phi_{\beta-\eps_0}^{\beta+\eps_0}=\cK^\beta.
	\end{equation}
	Indeed, observe that each orbit $G\ast u$ for $u\in\cK$ corresponds to a single value of $\Phi$ and since the number of distinct critical orbits is finite,
	$\cK^\beta$ is empty except for finitely many values $\beta$.\\
	{\em Claim 2.} For every $\delta\in (0,m_{\beta+\eps_0})$ there is $\eps\in (0,\eps_0]$ such that
	\begin{equation}\label{eq:entrancetime1}
	\lim_{t\to T(u)} \Phi(\eta(t,u)) < \beta -\eps \quad\hbox{for } u\in \Phi^{\beta+\eps}_{\beta-\eps_0}\setminus B(\cK^\beta,\delta),
	\end{equation}
	which means that if $u$ is away from $\cK^\beta$ at a level not greater than $\beta+\eps$, then the the flow $\eta$ brings $u$ below the level $\beta$. This will be used to define the entrance time map below; see {\em Conclusion}.
	Take $u\in \Phi^{\beta+\eps}_{\beta-\eps_0}\setminus \cK^\beta$ and
	observe that by (A2),
	$\Phi(\eta(t,u))=J(m_\cU(\eta(t,u)))$  is bounded from below by $c$, and by  \eqref{eq:flow3} it is decreasing in $t\in [0,T(u))$. Hence $\lim_{t\to T(u)} \Phi(\eta(t,u))$ exists.
	Suppose that 
	\eqref{eq:entrancetime1} does not hold, i.e. there is $\delta\in (0,m_{\beta+\eps_0})$ such that for any $\eps\in (0,\eps_0]$ 
	$$A_\eps:=\big\{u\in \Phi^{\beta+\eps}_{\beta-\eps_0}\setminus B(\cK^\beta,\delta):  
	\lim_{t\to T(u)} \Phi(\eta(t,u)) \geq \beta -\eps
	\big\}\neq \emptyset.$$
	We will now show that for any $u\in A_{\eps_0}$
	\begin{equation}\label{eq:Claim1}
	\lim_{t\to T(u)} \inf_{v\in \cK^\beta}\|\eta(t,u)-v\|=\lim_{t\to T(u)} \mathrm{dist} \big(\eta(t,u),\cK^\beta\big) =0.
	\end{equation}
Let us start to show  that for $u\in A_{\eps_0}$, $\lim_{t\to T(u)}\eta(t,u)$ exists.
	Suppose that, on the contrary, there is $0<\eta_0<m_{\beta+\eps_0}$ and there is an increasing sequence $(t_n)\subset [0,T(u))$ such that $t_n\to T(u)$ and 
	\begin{eqnarray}\label{eq:eps_0}
	\eta_0&<&\|\eta(t_{n+1},u)-\eta(t_n,u)\|<m_{\beta+\eps_0}
	\end{eqnarray} 
	for $n\geq 1$.
	Note that by \eqref{eq:flow2} and \eqref{eq:flow3}
	\begin{eqnarray*}
		\eta_0<\|\eta (t_{n+1},u)-\eta(t_n,u)\|&\leq& \int_{t_n}^{t_{n+1}}\|v(\eta(s,u))\|\, ds\leq 2 \int_{t_n}^{t_{n+1}}\|\Phi'(\eta(s,u))\|\, ds\\
		&\leq& 2\int_{t_n}^{t_{n+1}}\big(\Phi'(\eta(s,u))(v(\eta(s,u)))\big)^{1/2}\, ds\\
		&\leq& 2\sqrt{t_{n+1}-t_n} \Big(\int_{t_n}^{t_{n+1}} \Phi'(\eta(s,u))(v(\eta(s,u)))\,ds\Big)^{1/2}\\
		&=& 2\sqrt{t_{n+1}-t_n} \big(\Phi(\eta(t_{n},u))-\Phi(\eta(t_{n+1},u))\big)^{1/2}\\
		&\leq & 2\sqrt{t_{n+1}-t_n} (\beta +\eps)^{1/2}.
	\end{eqnarray*}
	Hence $|t_{n+1}-t_n|\geq \frac{\eta_0^2}{4(\beta+\eps)}$ and $T(u)=\infty$.
	Again, by \eqref{eq:flow3}
	\begin{equation}\label{eq:PSetha}
	\int_{t_n}^{t_{n+1}}\|\Phi'(\eta(s,u))\|^2\, ds\leq \big(\Phi(\eta(t_n,u))-\Phi(\eta(t_{n+1},u))\big)\to 0
	\end{equation}
	as $n\to\infty$. Then, for every $n$,  we find $t_n'\in[t_n,t_{n+1}]$ such that $\Phi'(\eta(t_n',u))\to 0$ and by  \eqref{eq:eps_0} we may assume that $\eta_0<\|\eta(t_{n}',u)-\eta(t_{n+1}',u)\|<m_{\beta+\eps_0}$ for $n\geq 1$.
	 Therefore we get a contradiction with $(M)_{\beta+\eps_0}(ii)$. Hence $u_0=\lim_{t\to T(u)}\eta(t,u)$ exists and since $J(\eta(t,u))\leq J(u)$ is bounded as $t\to T(u)$, by (A3) we get $u_0\notin \partial\cU$. From
	the definition of $T(u)$, we infer that $u_0\in \cK$.  
Moreover, by \eqref{eq:KL}
	$$u_0\in \cK\cap \Phi^{\beta+\eps_0}_{\beta-\eps_0}=\cK^\beta,$$
which completes the proof of \eqref{eq:Claim1}.\\
Now observe that in view of \eqref{eq:Claim1}, for $u\in A_{\eps_0}$ we may define
	\begin{eqnarray*}
		t_0(u)&:=&\inf\big\{t\in [0,T(u)):  \eta(s,u)\in B(\cK^\beta,m_{\beta+\eps_0})\hbox{ for all }s> t\big\}\\
		t(u)&:=&\inf\big\{t\in [t_0(u),T(u)):  \eta(t,u)\in B(\cK^\beta,\delta/2)\big\}
	\end{eqnarray*}
	Note that $0\leq t_0(u)< t(u)<T(u)$ and we show that
	\begin{equation}\label{eq:inft(u)}
	\inf_{u\in A_{\eps_0}}
	t(u)-t_0(u)\geq \frac{\delta^2}{16(\beta+\eps_0)}.
	\end{equation}
	Indeed, if $u\in A_{\eps_0}$, then by \eqref{eq:flow2} and \eqref{eq:flow3} we have
	\begin{eqnarray*}
		\frac{\delta}{2}&\leq&
		\|\eta(t_0(u),u)-\eta(t(u),u)\| \leq \int_{t_0(u)}^{t(u)}\|v(\eta(s,u))\|\, ds
		\leq 2\int_{t_0(u)}^{t(u)}\|\Phi'(\eta(s,u))\|\, ds\\
		&\leq& 2\int_{t_0(u)}^{t(u)}\big(\Phi'(\eta(s,u))(v(\eta(s,u_n)))\big)^{1/2}\, ds\\
		&\leq& 
		2\sqrt{t(u)-t_0(u)}\Big(\int_{t_0(u)}^{t(u)}\Phi'(\eta(s,u))(v(\eta(s,u)))\, ds\Big)^{1/2}\\
		&=& 
		2\sqrt{t(u)-t_0(u)}\big(\Phi(\eta(t_0(u)u)-\Phi(\eta(t(u),u))\big)^{1/2}\\
		&\leq& 2\sqrt{t(u)-t_0(u)}(\beta+\eps_0)^{1/2}
	\end{eqnarray*}
	and we get \eqref{eq:inft(u)}. Note that $A_{\eps_0/2}\subset A_{\eps_0}$ and
	let
	$$\rho:=\inf_{u\in A_{\eps_0/2}}\int_{t_0(u)}^{t(u)}\|\Phi'(\eta(s,u)\|^2\, ds.$$
	If $\rho=0$ then by \eqref{eq:inft(u)} we find $u_n\in A_{\eps_0/2} $ and $t_n\in (t_0(u_n),t(u_n))$ such that $$\Phi'(\eta(t_n,u_n))\to 0 \hbox{ as } n\to\infty.$$
	Since $t_n> t_0(u_n)$ we have $\eta(t_n,u_n)\in B(\cK^\beta,m_{\beta+\eps_0})$ and passing to a subsequence $$\Phi(\eta(t_n,u_n))\to\alpha\leq \beta+\eps_0/2 < \beta+\eps_0.$$ In view of Lemma \ref{lem:M_betacond}, passing to a subsequence, we obtain
	$$g_n\eta(t_n,u_n)\to u$$
	for some $u\in \cK$ and $g_n\in G$. By \eqref{eq:KL} we get $u\in \cK^\beta$. On the other hand, since $t_n<t(u_n)$ we obtain
	$$g_{n}\eta(t_n,u_n)\notin B(\cK^\beta,\delta/2),$$
	which is a contradiction. Therefore $\rho>0$ and we take 
	$$\eps < \min\Big\{\frac12\eps_0,\frac14\rho\Big\}.$$
	Let $u\in A_\eps\subset A_{\eps_0/2}$ and 
	since
	\begin{eqnarray*}
		\Phi(\eta(t(u),u))-\Phi(\eta(t_0(u),u)) &=&
		-\int_{t_0(u)}^{t(u)}\Phi'(\eta(s,u))(v(\eta(s,u)))\,ds\\
		&\leq& 
		-\frac12\int_{t_0(u)}^{t(u)}\|\Phi'(\eta(s,u)\|^2\, ds,
	\end{eqnarray*}
	we obtain
	\begin{eqnarray*}
		\beta-\eps&\leq& \lim_{t\to T(u)} \Phi(\eta(t,u))\leq \Phi(\eta(t(u),u))\\
		&\leq&\beta +\eps -\frac12\int_{t_0(u)}^{t(u)}\|\Phi'(\eta(s,u)\|^2\, ds
		\leq \beta+\eps -\frac12\rho\\
		&<&\beta-\eps,
	\end{eqnarray*}
	which gives again a contradiction.
	Thus we have finally proved that \eqref{eq:entrancetime1} holds.\\
{\em Conclusion.}	Now take any $\delta<m_{\beta+\eps_0}$ such that
	$$\gamma (\cl B(\cK^\beta,\delta))=\gamma (\cK^\beta).$$
	Let us define the {\em entrance time map} $e:\Phi^{\beta+\eps}_{\beta-\eps_0}\setminus B(\cK^\beta,\delta)\to [0,\infty)$ such that
	$$e(u):=\inf\{t\in [0,T(u)):  \Phi(\eta(s,u))\leq \beta -\eps\}.$$
	It is standard to show that $e$ is continuous and even.
	Moreover we may define a continuous and odd map 
	$h:\Phi^{\beta+\eps}\setminus B(\cK^\beta,\delta)\to \Phi^{\beta-\eps}$
	such that 
	\begin{equation*}
	h(u)=\left\{
	\begin{array}{ll}
	\eta(e(u),u)&\quad \hbox{for }u\in \Phi^{\beta+\eps}_{\beta-\eps_0}\setminus B(\cK^\beta,\delta),\\
	u &\quad \hbox{for }u\in \Phi^{\beta-\eps_0}.
	\end{array}
	\right.
	\end{equation*}
	Let us take $\beta=\beta_k$ defined by \eqref{eq:LSvalue} for some $k\geq 1$. Then
	by  \cite[Proposition II.5.4 $(2^\circ)$ and $(4^\circ)$]{Struwe}
	$$\gamma(\Phi^{\beta+\eps}\setminus B(\cK^\beta,\delta))\leq 
	\gamma(\Phi^{\beta-\eps})\leq k-1$$
	and by  \cite[Proposition II.5.4 $(3^\circ)$]{Struwe}
	\begin{equation}\label{eq:LSvaluse}
	k\leq \gamma(\Phi^{\beta+\eps})\leq \gamma (\cl B(\cK^\beta,\delta))+
	\gamma(\Phi^{\beta+\eps}\setminus B(\cK^\beta,\delta))\leq \gamma(\cK^\beta)+k-1.
	\end{equation}
	Thus $\cK^\beta\neq \emptyset$, and since it has finite number of orbits and (A1) holds, we easy show that there is a continuous and odd map from $\cK^\beta$ with values in $\{-1,1\}$. Thus $\gamma(\cK^\beta)=1$. Note that if $\beta_k=\beta_{k+1}$ for some
	$k\geq 1$, then by \eqref{eq:LSvaluse} we get $\gamma(\cK^{\beta_k})\geq 
	\gamma(\Phi^{\beta_{k}+\eps}) -k = 	\gamma(\Phi^{\beta_{k+1}+\eps}) -k\geq 
	2$, which is a contradiction. Hence we get an infinite sequence
	$\beta_1<\beta_2<...$
	of critical values, which contradicts that $\cK$ consists of a finite number of distinct orbits. 
This completes the proof of (b).
\end{proof}

\begin{Rem}\label{rem:CrtiticalPointTheory}
(a) In this paper we consider the problem \eqref{eq} having the mountain pass geometry, hence we assume that $Y=X=H^1(\R^N)$, $\cM$ is a Pohozaev manifold in  $H^1(\R^N)$, $\cU$ is given by \eqref{PdefofU} and we show that $(M)_\beta\;(i)$ holds with translations $G=\R^N$; see Lemma \ref{lem:Mcond} below. Therefore we will apply  Theorem \ref{Th:CrticMulti} (a) in proof of Theorem \ref{ThMain1}. In order to prove Theorem \ref{ThMain2} we will also consider the setting of this section for $Y=X=X_\tau\cap H^1_{\cO_1}(\R^N)$ with $G=\{0\}\times\{0\}\times\R^{N-2M}$;  see Lemma \ref{lem:Mcond2} below. Moreover, we apply the multiplicity result contained in Theorem \ref{Th:CrticMulti} (c) and we show that $J\circ m_\cU$ satisfies the Palais-Smale condition in  $\cU\cap X_\tau\cap H^{1}_{\cO_2}(\R^N)$; see Lemma \ref{lem:Mcond3}.\\
(b) In this work, however,  we do not apply $(M)_\beta\;(ii)$  and Theorem \ref{Th:CrticMulti} (b) and we would like to mention other possible applications.  Note that for an indefinite functio\-nal $J$, i.e. when $0$ is not a local minimum, one needs to find a proper subspace $Y\subset X$ such that (A3) holds and the Palais-Smale condition may not be satisfied. For instance, as in \cite{SzulkinWeth,MederskiENZ} one can consider a generalized Nehari manifold 
$$\cM=\big\{u\in X\setminus Y': J'(u)(u)=0\hbox{ and }J'(u)|_{Y'}=0\big\},$$ 
where $Y'$ is an orthogonal complement of $Y$ in $X$.
The approaches considered in these works fit into the abstract setting of this section. For instance $X=H^1(\R^N)$, $Y$ and $Y'$ are subspaces, where the Schr\"odinger operator $-\Delta +V$ is positive and negative definite respectively, $0$ lies in a spectral gap of this operator and we can consider the $\Z^N$-periodic problem $-\Delta u + V(x)u=g(x,u)$ as in \cite{SzulkinWeth}. The discreteness of Palais-Smale sequences obtained in \cite[Lemma 2.14]{SzulkinWeth} implies $(M)_\beta\;(ii)$ with $G=\Z^N$ and we can reprove the results of \cite{SzulkinWeth}. Similarly one can recover the variational approach in \cite{MederskiENZ}.
Hence, Theorem \ref{Th:CrticMulti} may be applied to Pohozev as well as Nehari-type topological constraints.
\end{Rem}

\section{Concentration compactness and profile decompositions}\label{sec:Lions}

The following lemma is known if  $\Psi(s)=|s|^p$ with $2<p<2^*$ and is due to Lions \cite[Lemma 1.21]{Willem}, \cite{Lions84}. We will use it to prove the profile decomposition Theorem \ref{ThGerard} with a general function $\Psi$.

\begin{Lem}\label{lem:Lions}
	Suppose that $(u_n)\subset H^{1}(\R^N)$ is bounded and for some $r>0$ 	\begin{equation}\label{eq:LionsCond11}
	\lim_{n\to\infty}\sup_{y\in \R^N} \int_{B(y,r)} |u_n|^2\,dx=0.
	\end{equation}
	Then 
	$$\int_{\R^N} \Psi(u_n)\, dx\to 0\quad\hbox{as } n\to\infty$$
	for any continuous function $\Psi:\R\to [0,\infty)$ such that \eqref{eq:Psi2} holds.
\end{Lem}
\begin{proof}
	Take any $\eps>0$ and $2<p<2^*$ and suppose that $\Psi$ satisfies \eqref{eq:Psi2}. Then we find $0<\delta<M$ and $c_\eps>0$ such that 
	\begin{eqnarray*}
	\Psi(s)&\leq& \eps |s|^{2}\quad\hbox{ if }|s|\in [0,\delta],\\
	\Psi(s)&\leq& \eps |s|^{2^*}\quad\hbox{ if }|s|>M,\\
	\Psi(s)&\leq& c_\eps |s|^{p}\quad\hbox{ if }|s|\in (\delta,M].
	\end{eqnarray*}
Hence, in view of Lions lemma \cite{Willem}[Lemma 1.21] we get
$$\limsup_{n\to\infty}\int_{\R^N}\Psi(u_n)\, dx\leq \eps \limsup_{n\to\infty}\int_{\R^N}|u_n|^2+|u_n|^{2^*}\, dx.$$
Since $(u_n)$ is bounded in $L^2(\R^N)$ and in $L^{2^*}(\R^N)$, we conclude by letting $\eps\to0$.
\end{proof}

Let us consider $x=(x^1,x^2,x^3)\in\R^N=\R^M\times\R^M\times \R^{N-2M}$ with $2\leq M\leq N/2$ 
such that $x^1,x^2\in\R^M$ and $x^3\in\R^{N-2M}$. Let
$\cO_1:=\cO(M)\times \cO(M)\times\id\subset \cO(N)$ and now we consider invariant functions with respect to $\cO_1$.

\begin{Cor}\label{CorLions1}
	Suppose that $(u_n)\subset H^1_{\cO_1}(\R^N)$ is bounded and for all $r>0$
	\begin{equation}\label{eq:LionsCond12}
	\lim_{n\to\infty}\sup_{z\in \R^{N-2M}} \int_{B((0,0,z),r)} |u_n|^2\,dx=0.
	\end{equation}
	Then 
	$$\int_{\R^N} \Psi(u_n)\, dx\to 0\quad\hbox{as } n\to\infty$$
	for any continuous function $\Psi:\R\to [0,\infty)$ such that \eqref{eq:Psi2} holds.
\end{Cor}
\begin{proof}
	Suppose that
	\begin{equation}\label{eq:LionsCond12proof1}
	\int_{B(y_n,1)} |u_n|^2\,dx\geq c>0
	\end{equation}
	for some sequence $(y_n)\subset \R^N$ and a constant $c$.
	Observe that in the family $\{B(hy_n,1)\}_{h\in\cO_1}$ we find an increasing number of disjoint balls provided that $|(y^1_n,y^2_n)|\to\infty$. Since $(u_n)$ is bounded in $L^2(\R^N)$ and invariant with respect to $\cO_1$, by \eqref{eq:LionsCond12proof1} $|(y^1_n,y^2_n)|$ must be bounded. Then for sufficiently large $r\geq r_0$ one obtains
	$$\int_{B((0,0,y_n^3),r)} |u_n|^2\,dx\geq \int_{B(y_n,1)} |u_n|^2\,dx\geq  c>0,$$
	and we get a contradiction with \eqref{eq:LionsCond12}. Therefore  \eqref{eq:LionsCond11} is satisfied with $r=1$ and by Lemma \ref{lem:Lions} we conclude.
\end{proof}

\begin{Rem} Instead of $\cO_1$ in Corollary \ref{CorLions1} one can consider any subgroup $G=\cO'\times\id\subset \cO(N)$ such that $\cO'\subset \cO(M)$ and $\R^M$ is compatible with $\cO'$ for some $0\leq M\leq N$
(in the sense of  \cite[Definition 1.23]{Willem}, cf. \cite{Lions82}), i.e. if
$\lim_{|y|\to\infty, y\in\R^M}m(y,r)=\infty$ for some $r>0$,
where
$$m(y,r):=\sup\big\{n\in\N: \hbox{there exist }g_1,...,g_n\in\cO'\hbox{ such that } B(g_iy,r)\cap B(g_jy,r)=\emptyset\hbox{ for }i\neq j\big\}.$$
\end{Rem}

Now let us assume in addition that  $N-2M\neq 1$ and
$$\cO_2:=\cO(M)\times \cO(M)\times\cO(N-2M)\subset \cO(N).$$
In view of \cite{Lions82}, $H^1_{\cO_2}(\R^N)$ embeds compactly into $L^p(\R^N)$ for $2<p<2^*$. In order to deal with the general nonlinearity we need the following result.

\begin{Cor}\label{CorLions2}
	Suppose that $(u_n)\subset H^1_{\cO_2}(\R^N)$ is bounded and $u_n\to 0$ in $L^2_{loc}(\R^N)$.
	Then
	$$\int_{\R^N} \Psi(u_n)\, dx\to 0\quad\hbox{as } n\to\infty$$
	for any continuous function $\Psi:\R\to [0,\infty)$ such that \eqref{eq:Psi2} holds.
\end{Cor}
\begin{proof}
	Observe that for all $r>0$
	$$\lim_{n\to\infty}\int_{B(0,r)} |u_n|^2\,dx=0$$
	and similarly as in proof of Corollary \ref{CorLions1} we complete the proof.
\end{proof}

\begin{altproof}{Theorem \ref{ThGerard}}
	Let $(u_n)\subset H^1(\R^N)$ be a bounded sequence and $\Psi$ as in Theorem \ref{ThGerard}.
	We claim that, passing to a subsequence, there is $K\in \N\cup \{\infty\}$ and there is a sequence
	$(\tu_i)_{i=0}^K\subset H^1(\R^N)$, for $0\leq i <K+1$$(\footnote{If $K=\infty$ then $K+1=\infty$ as well.})$  there are sequences $(v_n^i)\subset H^1(\R^N)$, $(y_n^i)\subset \R^N$ and positive numbers $(c_i)_{i=0}^{K}, (r_i)_{i=0}^{K}$ such that $y_n^0=0$, $r_0=0$ and for any  $0\leq i<K+1$ one has
	\begin{eqnarray}	\label{Eqxnxm1}
	&&u_n(\cdot+y_n^i)\weakto\tu_i\hbox{ in }H^1(\R^N)\hbox{ and }u_n(\cdot+y_n^i)\chi_{B(0,n)}\to\tu_i\hbox{ in }L^{2}(\R^N)\hbox{ as }n\to\infty,\\	\label{Eqxnxm2}
	&&\tu_i\neq 0\hbox{ if }i\geq 1,\\
	\label{Eqxnxm}
	&&|y_n^i-y_n^j|\geq n-r_i-r_j\hbox{ for } 0\leq j\neq i< K+1 \hbox{ and sufficeintly large }n,\\	\label{Eqxnxm3}
	&& v_n^{-1}:=u_n\hbox{ and }v_n^i:=v_n^{i-1}-\tu_i(\cdot-y_n^i)\hbox{ for }n\geq 1\\\label{EqIntegralunSumci}
	&&\int_{B(y_n^{i},r_i)}|v_n^{i-1}|^2\, dx \geq c_{i}\geq\frac{1}{2}\sup_{y\in\R^N}\int_{B(y,r_i)}|v_n^{i-1}|^2\, dx\hbox{ for sufficienlty large }n,\\
	&&	r_i\geq \max\{i,r_{i-1}\},\hbox{if }i\geq 1,\hbox{ and } c_i= \frac{3}{4}\lim_{r\to\infty}\limsup_{n\to\infty}\sup_{y\in\R^N}\int_{B(y,r)}|v_n^{i-1}|^2\,dx\nonumber
	>0.
	\end{eqnarray}
	Moreover \eqref{EqSplit2a} is satisfied. 
	Since $(u_n)$ is bounded, 
	passing to a subsequence we may assume that $\lim_{n\to\infty}\int_{\R^N}|\nabla u_n|^2\,dx$ exists and
	\begin{eqnarray*}
		u_n &\weakto& \tu_0\quad \hbox{ in }H^1(\R^N),\\
		u_n\chi_{B(0,n)} &\to& \tu_0\quad \hbox{ in }L^2(\R^N).
	\end{eqnarray*}
	The latter convergence follows from the fact that for any $n$, $H^1(B(0,n))$ is compactly embedded into $L^2(B(0,n))$ and	
	we find sufficiently large $k_n$ such that $|(u_{k_n}-\tu_0)\chi_{B(0,n)}|_2<\frac1n$, where $\chi_{B(0,n)}$ is the characteristic function of $B(0,n)$ and  $|\cdot|_p$ denotes the usual $L^p$-norm  for $p\geq 1$. The subsequence $(u_{k_n})$ is then relabelled by $(u_{n})$.
	Take $v_n^0:=u_n-\tu_0$ and
	if
	\begin{equation*}
	\lim_{n\to\infty}\sup_{y\in\R^N}\int_{B(y,r)}|v_n^0|^2\, dx=0
	\end{equation*}
	for every $r\geq 1$,
	then we finish the proof of our claim with $K=0$.
	Otherwise we get
	$$\infty>\sup_{n\geq 1}\int_{\R^N}|v_n^0|^2\,dx\geq c_1:=\frac34\lim_{r\to\infty}\limsup_{n\to\infty}\sup_{y\in\R^N}\int_{B(y,r)}|v_n^0|^2\, dx>0,$$
	and there is $r_1\geq 1$ and, passing to a subsequence, we find $(y_n^1)\subset\R^N$  such that
	\begin{equation}\label{eq:LemProofLions1}
	\int_{B(y_n^{1},r_1)}|v_n^0|^2\, dx \geq c_{1}\geq \frac{1}{2}\sup_{y\in\R^N} \int_{B(y,r_1)}|v_n^0|^2\, dx.
	\end{equation}
	Note that $(y_n^1)$ is unbounded and we may assume that $|y_n^1|\geq n-r_1$.  Since $(u_n(\cdot+y_n^1))$ is bounded in $H^1(\R^N)$, up to a subsequence, we find $\tu_1\in H^1(\R^N)$ such that 
	$$u_n(\cdot+y_n^1)\weakto \tu_1\quad\hbox{in }H^1(\R^N).$$ In view of \eqref{eq:LemProofLions1},  we get $\tu_1\neq 0$, and again we may assume that $u_n(\cdot+y_n^1)\chi_{B(0,n)}\to \tu_1$ in $L^2(\R^N)$.  
	Since
	$$\lim_{n\to\infty}\Big(\int_{\R^N}|\nabla (u_n-\tu_0)(\cdot +y_n^1)|^2\, dx-\int_{\R^N}|\nabla v_n^1(\cdot +y_n^1)|^2\,dx\Big)=\int_{\R^N}|\nabla \tu_1|^2\, dx,$$
	where  $v_n^1:=v_n^0-\tu_1(\cdot -y_n^1)=u_n-\tu_0-\tu_1(\cdot -y_n^1)$ ,
	then
	\begin{eqnarray*}
	\lim_{n\to\infty}\int_{\R^N}|\nabla u_n|^2\, dx &=&\int_{\R^N}|\nabla \tu_0|^2\,dx+\int_{\R^N}|\nabla \tu_1|^2\, dx+\lim_{n\to\infty}\int_{\R^N}|\nabla v_n^1|^2\, dx
	\end{eqnarray*}
	If 
	\begin{equation*}
	\lim_{n\to\infty}\sup_{y\in\R^N}\int_{B(y,r)}|v_n^1|^2\, dx=0
	\end{equation*}
	for every $r\geq \max\{2,r_1\}$,
	then we finish the proof of our claim with $K=1$. Otherwise,  $c_2:=\frac{3}{4}\lim_{r\to\infty}\limsup_{n\to\infty}\sup_{y\in\R^N}\int_{B(y,r)}|v_n^{1}|^2\,dx>0,$ 
	there is $r_2\geq \max\{2,r_1\}$ and, passing to a subsequence, we find  $(y_n^2)\subset\R^N$ such that
	\begin{equation}\label{eq:LemProofLions2}
	\int_{B(y_n^{2},r_2)}|v_n^1|^2\, dx \geq c_{2}\geq \frac{1}{2}\sup_{y\in\R^N} \int_{B(y,r_2)}|v_n^1|^2\, dx
	\end{equation}
	and $|y_n^2|\geq n-r_2$. Moreover $|y_n^2-y_n^1|\geq n-r_2-r_1$. Otherwise $B(y_n^2,r_2)\subset B(y_n^1,n)$ and the convergence $u_n(\cdot+y_n^1)\chi_{B(0,n)}\to \tu_1$ in $L^2(\R^N)$ contradicts \eqref{eq:LemProofLions2}.
	Then, passing to a subsequence, we find $\tu_2\neq 0$ such that 
	$$v_n^1(\cdot+y_n^2),\;u_n(\cdot +y_n^2)\weakto \tu_2\text{ in }H^1(\R^N)\hbox{ and }u_n(\cdot+y_n^2)\chi_{B(0,n)}\to \tu_2\hbox{ in } L^2(\R^N).$$ Again, if 
	\begin{equation*}
	\lim_{n\to\infty}\sup_{y\in\R^N}\int_{B(y,r)}|v_n^2|^2\, dx=0,
	\end{equation*}
	for every $r\geq \max\{3,r_2\}$, where $v_n^2:=v_n^1-\tu_2(\cdot-y_n^2)$,
	then we finish proof with $K=2$. Continuing the above procedure, for each $i\geq 1$ we find a subsequence of $(u_n)$, still denoted by $(u_n)$, such that \eqref{Eqxnxm1}--\eqref{EqIntegralunSumci} and \eqref{EqSplit2a} are satisfied. Similarly as above, if there is $i\geq 0$ such that
	\begin{equation}\label{eq:Kfinite}
	\lim_{n\to\infty}\sup_{y\in\R^N}\int_{B(y,r)}|v_n^i|^2\, dx=0
	\end{equation}
   for every $r\geq \max\{n,r_{i-1}\}$, then $K:=i$ and we finish proof of the claim.
Otherwise, $K=\infty$, by a standard diagonal method and passing to a subsequence, we show that  \eqref{Eqxnxm1}--\eqref{EqIntegralunSumci} and \eqref{EqSplit2a} are satisfied for every $i\geq 0$.\\
\indent 
	Now we show that \eqref{EqSplit3a} holds. 
	Observe that
	\begin{eqnarray*}
	\lim_{n\to\infty}\int_{\mathbb{R}^N}\big(\Psi(u_n)-\Psi(v_n^0)\big)\, dx &=&\int_{\mathbb{R}^N}\Psi(\tu_0)\, dx.
	\end{eqnarray*}
	Indeed, by Vitali's convergence theorem 
	\begin{eqnarray*}
		\int_{\mathbb{R}^N}\big(\Psi(u_n)-\Psi(v_n^0)\big)\, dx
		&=&\int_{\mathbb{R}^N}\int_0^1 -\frac{d}{ds}\Psi(u_n-s\tu_0)\, ds\,dx\\\nonumber
		&=&\int_{\mathbb{R}^N}\int_0^1 \Psi'(u_n-s\tu_0)\tu_0\,ds\, dx\\\nonumber
		&\rightarrow& \int_0^1 \int_{\mathbb{R}^N} \Psi'(\tu_0-s\tu_0)\tu_0\,dx\,ds\\\nonumber
		&=&\int_{\mathbb{R}^N}\int_0^1 -\frac{d}{ds}\Psi(\tu_0-s\tu_0)\, ds\, dx\\
		&=&\int_{\mathbb{R}^N}\Psi(\tu_0)\, dx\nonumber
	\end{eqnarray*}
	as $n\to\infty$. 
	Then 
	\begin{equation}\label{eq:Vit1}
	\limsup_{n\to\infty}\int_{\mathbb{R}^N}\Psi(u_n)\, dx =\int_{\mathbb{R}^N}\Psi(\tu_0)\, dx+\limsup_{n\to\infty}\int_{\R^N}\Psi(v_n^0)\, dx
	\end{equation}
	and \eqref{EqSplit3a} holds for $i=0$.
	Similarly as above 
	we show that
	$$\lim_{n\to\infty}\int_{\mathbb{R}^N}\big(\Psi((u_n-\tu_0)(\cdot +y_n^1))-\Psi(v_n^1(\cdot +y_n^1))\big)\, dx=\int_{\mathbb{R}^N}\Psi(\tu_1)\, dx.$$
	In view of \eqref{eq:Vit1} we obtain
	\begin{eqnarray*}
		\limsup_{n\to\infty}\int_{\mathbb{R}^N}\Psi(u_n)\, dx &=&\int_{\mathbb{R}^N}\Psi(\tu_0)\, dx+\limsup_{n\to\infty}\int_{\R^N}\Psi(u_n-\tu_0)\, dx\\	
		&=&\int_{\mathbb{R}^N}\Psi(\tu_0)\, dx+\int_{\mathbb{R}^N}\Psi(\tu_1)\, dx+\limsup_{n\to\infty}\int_{\R^N}\Psi(v_n^1)\, dx.
	\end{eqnarray*}
	Continuing the above procedure we prove that \eqref{EqSplit3a} holds for every $i\geq 0$. Now observe that, if there is $i\geq 0$ such that \eqref{eq:Kfinite} holds
	for every $r\geq \max\{i,r_i\}$,
	then $K=i$. If, in addition, \eqref{eq:Psi2} holds, then in view of Lemma \ref{lem:Lions} we obtain that
	$$\lim_{n\to\infty}\int_{\R^N}\Psi(v_n^i)\, dx=0$$
	and we finish the proof by setting $\tu_j=0$ for $j>i$. 
	Otherwise we have $K=\infty$. It remains to prove \eqref{EqSplit4a} in this case. Note that by \eqref{EqIntegralunSumci} we have
	\begin{eqnarray*}
		c_{k+1}&\leq& \int_{B(y_n^{k+1},r_{k+1})}
		|v_n^k|^2\, dx\\
		&\leq&2\int_{B(y_n^{k+1},r_{k+1})}|v_n^i|^2\, dx+2\int_{B(y_n^{k+1},r_{k+1})}\Big|\sum_{j=i+1}^{k}\tu_j(\cdot -y_n^j)\Big|^2\, dx\\
		&\leq&2\sup_{y\in\R^N}\int_{B(y,r_{k+1})}|v_n^{i}|^2\,dx+ 2(k-i)\sum_{j=i+1}^{k}\int_{B(y_n^{k+1}-y_n^j,r_{k+1})}|\tu_j|^2\, dx
	\end{eqnarray*}
	for any $0\leq  i<k$.
	Taking into account \eqref{Eqxnxm} and letting $n\to\infty$ we get $c_{k+1}\leq \frac83 c_{i+1}$. Take $k\geq 1$ and  sufficiently large $n>4r_{k}$ such that \eqref{EqIntegralunSumci} and \eqref{Eqxnxm} are satisfied. Then we obtain
	\begin{equation*}
	\begin{aligned}
	\frac{3}{32}\sup_{y\in\R^N}\int_{B(y,r_{k+1})}|v_n^k|^2\,dx
	&\leq \frac{3}{16} c_{k+1}\leq \frac{1}{2k} \sum_{i=0}^{k-1} c_{i+1}\leq\frac{1}{2k} \sum_{i=0}^{k-1}
	\int_{B(y_n^{i+1},r_{i+1})}|v_n^i|^2\, dx\\
	&\leq\frac{1}{k} \sum_{i=0}^{k-1}
	\int_{B(y_n^{i+1},r_{i+1})}\Big(|u_n|^2 + \Big|\sum_{j=0}^{i}\tu_j(\cdot -y_n^j)\Big|^2\Big)\, dx\\
	&= \frac{1}{k}
	\int_{\bigcup_{i=0}^{k-1}B(y_n^{i+1},r_{i+1})}|u_n|^2\,dx +\frac1k\int_{\R^N} \Big|\sum_{i=0}^{k-1}\sum_{j=0}^{i}\tu_j(\cdot -y_n^j)\chi_{B(y_n^{i+1},r_{i+1})}\Big|^2\, dx\\
	&\leq \frac{1}{k}
	|u_n|_2^2+\frac1k\Big|\sum_{i=0}^{k-1}\sum_{j=0}^{i}\tu_j(\cdot -y_n^j)\chi_{B(y_n^{i+1},r_{i+1})}\Big|^2_2.
	\end{aligned}
	\end{equation*}
	Observe that by \eqref{Eqxnxm} and since $n>4r_{k}$ we have
	$$B(y_n^{i+1}-y_n^j,r_{i+1})\subset \R^N\setminus B(0,n-3r_k)\hbox{ for }0\leq j< i<k$$
	and
	\begin{eqnarray*}
		\Big|\sum_{i=0}^{k-1}\sum_{j=0}^{i}\bar{u}_j(\cdot -y_n^j)\chi_{B(y_n^{i+1},r_{i+1})}\Big|_{2}&\leq& \nonumber
		\sum_{i=0}^{k-1}\sum_{j=0}^{i}\big|\bar{u}_j\chi_{B(y_n^{i+1}-y_n^j,r_{i+1})}\big|_{2}\leq \sum_{i=0}^{k-1}\sum_{j=0}^{i}\big|\bar{u}_j\chi_{\R^N\setminus B(0,n-3r_k)}\big|_{2}\\
		&\leq& k\sum_{j=0}^{k-1}\big|\bar{u}_j\chi_{\R^N\setminus B(0,n-3r_k)}\big|_{2}\to 0
	\end{eqnarray*}
	as $n\to\infty$. Hence
	\begin{equation}\label{eq:ineqGer1}
	\limsup_{n\to\infty}\Big(\sup_{y\in\R^N}\int_{B(y,r_{k+1})}|v_n^k|^2\,dx\Big)\leq  \frac{32}{3k}
	\limsup_{n\to\infty}|u_n|_2^2,
	\end{equation}
and
	suppose that \eqref{EqSplit4a} does not holds, that is
	\begin{equation}\label{eq:ineqGer2}
	\limsup_{i\to\infty} \Big(\limsup_{n\to\infty} \int_{\R^N}\Psi(v_n^i)\, dx\Big)>\delta
	\end{equation}
	for some $\delta>0$. Then we find increasing sequences $(i_k), (n_k)\subset \N$ such that
	$$\int_{\R^N}\Psi(v_{n_k}^{i_k})\, dx>\delta$$
	and 
	$$\sup_{y\in\R^N}\int_{B(y,r_{k+1})}|v_{n_k}^{i_k}|^2\,dx\leq \limsup_{n\to\infty}\Big(\sup_{y\in\R^N}\int_{B(y,r_{k+1})}|v_n^{i_k}|^2\,dx\Big)+\frac{1}{i_k}.$$
	Since \eqref{eq:ineqGer1} holds, we get
	$$\lim_{k\to\infty} \Big(\sup_{y\in\R^N}\int_{B(y,r_{k+1})}|v_{n_k}^{i_k}|^2\,dx\Big)=0,$$
	and in view of Lemma \ref{lem:Lions} we obtain that
	$$\lim_{k\to\infty} \int_{\R^N}\Psi(v_{n_k}^{i_k})\, dx=0,$$
	which is a contradiction. Hence \eqref{EqSplit4a}  is satisfied.			
\end{altproof}

Now we observe that in Theorem \ref{ThGerard} we may find translations $(y_n^i)_{i=0}^\infty\subset  \{0\}\times\{0\}\times\R^{N-2M}$ provided that
$(u_n)\subset H^{1}_{\cO_1}(\R^N)$ and $2\leq m< N/2$ .
\begin{Cor}\label{CorGerard}
	Suppose that $(u_n)\subset H^{1}_{\cO_1}(\R^N)$ is bounded and $2\leq M< N/2$.
	Then there are sequences
	$(\tu_i)_{i=0}^\infty\subset H^1_{\cO_1}(\R^N)$, $(y_n^i)_{i=0}^\infty\subset  \{0\}\times\{0\}\times\R^{N-2M}$ for any $n\geq 1$, such that the statements of Theorem \ref{ThGerard} are satisfied.
\end{Cor}
\begin{proof}
A careful inspection of proof of Theorem \ref{ThGerard} leads to the following claim: there is $K\in \N\cup \{\infty\}$ and there is a sequence
$(\tu_i)_{i=0}^K\subset H^1(\R^N)$, for $0\leq i <K+1$ there are sequences $(v_n^i)\subset H^1_{\cO_1}(\R^N)$, $(y_n^i)\subset \{0\}\times\{0\}\times\R^{N-2M}$ and positive numbers $(c_i)_{i=0}^{K}, (r_i)_{i=0}^{K}$ such that $y_n^0=0$, $r_0=0$ and, up to a subsequence, for any $n$ and $0\leq i<K+1$ one has
\eqref{Eqxnxm1}-\eqref{Eqxnxm3},
\begin{eqnarray*}
&&\int_{B(y_n^{i},r_i)}|v_n^{i-1}|^2\, dx \geq c_{i}\geq\frac{1}{2}\sup_{y\in\R^{N-2M}}\int_{B((0,0,y),r_i)}|v_n^{i-1}|^2\, dx\\
&&\hspace{4.1cm}	\geq \frac{1}{4}\sup_{r>0,y\in\R^{N-2M}} \int_{B((0,0,y),r)}|v_n^{i-1}|^2\, dx\nonumber
>0, r_i\geq \max\{i,r_{i-1}\} \hbox{ for }i\geq 1,
\end{eqnarray*}
and \eqref{EqSplit2a}, \eqref{EqSplit3a} are satisfied. In order to prove \eqref{EqSplit4a} we use Corollary \ref{CorLions1} instead of Lemma \ref{lem:Lions}. 
\end{proof}

Observe that if $m=N/2$, then we consider $\cO_2$-invariant sequences. In general we assume that  $N-2M\neq 1$ and we have the following result.

\begin{Cor}\label{cor:LionsO2}
Suppose that $(u_n)\subset H^{1}_{\cO_2}(\R^N)$ is bounded.
Then passing to a subsequence we find $\tu_0\in H^{1}_{\cO_2}(\R^N)$ such that
\begin{eqnarray*}
\nonumber
&& u_n\weakto \tu_0\; \hbox{ in } H^1(\R^N)\text{ as }n\to\infty,\\
&& \lim_{n\to\infty}\int_{\R^N}|\nabla u_n|^2\, dx= \int_{\R^N}|\nabla\tu_0|^2\, dx+\lim_{n\to\infty}\int_{\R^N}|\nabla (u_n-\tu_0)|^2\, dx,\\
&& \limsup_{n\to\infty}\int_{\R^N}\Psi(u_n)\, dx= 
\int_{\R^N}\Psi(\tu_0)\, dx+\limsup_{n\to\infty}\int_{\R^N}\Psi(u_n-\tu_0)\, dx	
\end{eqnarray*}
for any function $\Psi:\R\to[0,\infty)$ of class $\cC^1$ such that $\Psi'(s)\leq C(|s|+|s|^{2^*-1})$ for any $s\in\R$ and some constant $C>0$.
Moreover, if $\Psi$ satisfies \eqref{eq:Psi2},
then
\begin{equation*}
\lim_{n\to\infty}\int_{\R^N}\Psi(u_n-\tu_0)\, dx=0
\end{equation*}
and if $s\mapsto |\Psi'(s)s|$ satisfies \eqref{eq:Psi2}, then
\begin{equation}\label{EqPsiLast}
\lim_{n\to\infty}\int_{\R^N}\Psi'(u_n)u_n\, dx=\int_{\R^N}\Psi'(\tu_0)\tu_0\, dx.
\end{equation}
\end{Cor}
\begin{proof}
Similarly as in proof of Theorem \ref{ThGerard} we show that passing to a subsequence\\ $\lim_{n\to\infty}\int_{\R^N}|\nabla u_n|^2\,dx$ exists and 
\eqref{eq:Vit1} holds. Due to the compact embedding of $H^1(\R^N)$ into $L^2_{loc}(\R^N)$ we may assume that  $u_n-\tu_0\to 0$ in $L^2_{loc}(\R^N)$. Then we apply Corollary \ref{CorLions2} instead of Lemma \ref{lem:Lions}. In order to prove \eqref{EqPsiLast} observe that 
\begin{eqnarray*}
\int_{\R^N}|\Psi'(u_n)u_n-\Psi'(\tu_0)\tu_0|\, dx
&\leq& \int_{\R^N}|\Psi'(u_n)||u_n-\tu_0|\, dx+\int_{\R^N}|\Psi'(u_n)-\Psi'(\tu_0)||\tu_0|\, dx.
\end{eqnarray*}
Take any $\eps>0$, $2<p<2^*$ and we find $0<\delta<M$ and $c_\eps>0$ such that 
\begin{eqnarray*}
	|\Psi'(s)|&\leq& \eps (|s|+|s|^{2^*-1})\quad\hbox{ if }|s|\in [0,\delta]\hbox{ or }|s|>M,\\
	|\Psi'(s)|&\leq& c_\eps |s|^{2^*(1-\frac1p)}\quad\hbox{ if }|s|\in (\delta,M].
\end{eqnarray*}
Then, passing to a subsequence, $u_n\to \tu_0$ in $L^p(\R^N)$ and
we infer that 
$\int_{\R^N}|\Psi'(u_n)||u_n-\tu_0|\, dx\to 0$
 and
by the Vitali's convergence theorem  $ \int_{\R^N}|\Psi'(u_n)-\Psi'(\tu_0)||\tu_0|\, dx\to 0$ as $n\to\infty$.
\end{proof}

\section{Proofs of Theorems \ref{ThMain1},  \ref{ThMain2} and  \ref{ThMain3}}\label{sec:proof}

Let us consider the standard norm of $u$ in $H^1(\R^N)$ given by
$$\|u\|^2=\int_{\R^N}|\nabla u|^2+|u|^2\, dx.$$
In view of \cite{BerLions}[Theorem A.VI], $J:H^1(\R^N)\to \R$ given by \eqref{eq:action}, i.e. $J(u)=\frac12\int_{\R^N} |\nabla u|^2\,dx - \int_{\R^N} G(u)\, dx$
 is of class $\cC^1$. 
In the next subsections we build the variational setting according to Section \ref{sec:criticaltheory}.

\subsection{Critical point theory setting}
Let $X=Y=H^1(\R^N)$ and let $M,\psi:H^1(\R^N)\to \R$ be given by 
$$M(u)=\int_{\R^N}|\nabla u|^2\, dx-2^*\int_{\R^N}G(u)\, dx,\hbox{ and }\psi(u)=\int_{\R^N}|\nabla u|^2\, dx\quad\hbox{ for }u\in H^1(\R^N).$$ 
\begin{Prop}\label{prop:defMPSU}
	Let us denote
	\begin{eqnarray*}
		\cM&:=&\Big\{u\in H^1(\R^N): M(u)=0\Big\},\\
		\cS&:=&\Big\{u\in H^1(\R^N): \psi(u)=1\Big\},\\
		\cP&:=&\Big\{u\in H^1(\R^N): \int_{\R^N}G(u)\, dx>0\Big\},\\
		\cU&:=&\cS\cap\cP.
	\end{eqnarray*}
	Then the following holds.\\
	(i) There is a continuous map $m_\cP:\cP\to \cM$ such that $m_\cP(u)(x)=u(rx)$ for $x\in \R^N$ with 
	\begin{equation}\label{eq:defOfR}
	r=r(u)=\Big(\frac{2^*\int_{\R^N}G(u)\, dx}{\psi(u)}\Big)^{1/2}>0.
	\end{equation}
	(ii)  $m_\cU:=m_\cP|_{\cU}:\cU\to\cM$ is a homeomorphism with the inverse $m^{-1}(u)=u(\psi(u)^{\frac{1}{N-2}} \cdot)$, $J\circ m_{\cP}:\cP\to\R$ is of class $\cC^1$ with
	\begin{eqnarray*}
	(J\circ m_\cP)'(u)(v)&=&J'(m_\cP(u))(v(r(u)\cdot))\\
	&=&r(u)^{2-N}\int_{\R^N}\langle \nabla u,  \nabla v\rangle\,dx-r(u)^{-N}\int_{\R^N}g(u)v\, dx
	\end{eqnarray*}
	for $u\in\cP$ and $v\in H^1(\R^N)$.\\
	(iii) $J$ is coercive on $\cM$, i.e. for $(u_n)\subset \cM$, $J(u_n)\to\infty$ as $\|u_n\|\to\infty$,  and
	\begin{equation}\label{eq:infJM}
	c:=\inf_{\cM} J>0.
	\end{equation}
	(iv) If $u_n\to u$, $u_n\in \cU$ and $u\in\partial\cU=\big\{u\in\cS:\int_{\R^N}G(u)\,dx=0\big\}$, where the boundary of $\cU$ is taken in $\cS$, then $(J\circ m_\cU)(u)\to\infty$ as $n\to\infty$.
\end{Prop}
\begin{proof}
(i)	If $u\in \P$ then 
	\begin{eqnarray*}
		M(u(r\cdot)) &=& r^{-N}\Big(r^{2}\int_{\R^N} |\nabla u|^2\, dx - 2^*\int_{\R^N} G(u)\,dx\Big)=0
	\end{eqnarray*}
	for $r=r(u)$ given by \eqref{eq:defOfR}.
	Let $m_{\cP}:\P\to\cM$ be a map such that 
	$$m_{\cP}(u):=u(r(u)\cdot).$$
	Let $u_n\to u_0$, $u_n\in \cP$ for $n\geq 0$. 
	Observe that $r(u_n)\to r(u_0)$ and
	\begin{eqnarray*}
		\psi(m_\cP(u_n)-m_\cP(u_0)) &=& \int_{\R^N} \big| \nabla \big(u_n(r(u_n)\cdot )-u_0(r(u_0)\cdot)\big)\big|^2\, dx\\
		&\leq & 2r(u_n)^{2-N}\psi(u_n-u_0)+ 2\int_{\R^N} \big| \nabla \big(u_0(r(u_n)\cdot)-u_0(r(u_0)\cdot)\big)\big|^2\, dx\\
		&\leq & 2r(u_n)^{2-N}\psi(u_n-u_0) + 2\big(r(u_n)^{2-N}-r(u_0)^{2-N}\big)\psi(u_0)\\
		&&+4r(u_0)^{2-N}
		\int_{\R^N}\Big\langle \nabla u_0-\nabla u_0\Big(\frac{r(u_n)}{r(u_0)}\cdot\Big),\nabla u_0\Big\rangle\, dx\\
		&\to& 0
	\end{eqnarray*}
	passing to a subsequence. Similarly we show that $m_{\cP}(u_n)\to m_{\cP}(u_0)$ in $L^2(\R^N)$, hence 
	$m_\cP$ is continuous.\\ 
(ii) Observe that
	$$m_\cP^{-1}(u):=\{v\in \P: m_\cP(v)=u\}=\{u_\lambda: u_\lambda=u(\lambda \cdot),\; \lambda>0\}.$$
Then
$1=\psi(u_\lambda)=\lambda^{2-N}\psi(u)$
if and only if $\lambda = \psi(u)^{\frac{1}{N-2}}$. Therefore $m^{-1}(u)=u( \psi(u)^{\frac{1}{N-2}}\cdot)\in \cU$. Similarly as in (i) we show the continuity of $m^{-1}:\cM\to\cU$.
	Moreover for $u\in\cP$ and $v\in X$ one obtains
	$$
	 \begin{aligned}
		(J\circ m_{\cP})&'(u)(v)
		=\lim_{t\to 0}\frac{J(m_{\cP}(u+tv))-J(m_{\cP}(u))}{t}\\
		&=\lim_{t\to 0}\frac{(r(u+tv)^{2-N}-r(u)^{2-N})\int_{\R^N} |\nabla u|^2\, dx
			+r(u+tv)^{2-N}t\int_{\R^N} \langle \nabla(2u+tv),\nabla v\rangle\, dx}{2t}\\
		&\hspace{5mm}- \lim_{t\to 0}\frac{(r(u+tv)^{-N}-r(u)^{-N})\int_{\R^N} G(u)\,dx+r(u+tv)^{-N}\int_{\R^N}G(u+tv)-G(u)\, dx}{t}
			 \end{aligned}
				$$
			 		$$
			 		\begin{aligned}
		&\hspace{-15mm}=\frac{2-N}{2}r(u)^{1-N}r'(u)(v)\int_{\R^N}|\nabla u|^2\, dx+r(u)^{2-N}\int_{\R^N}\langle \nabla u, \nabla v\rangle\,dx\\
		&\hspace{-10mm}-\Big((-N)r(u)^{-N-1}r'(u)(v)\int_{\R^N}G(u)\, dx+r(u)^{-N}\int_{\R^N}g(u)v\, dx\Big)\\
		&\hspace{-15mm}=
		\frac{2-N}{2}r(u)^{-N-1}r'(u)(v)\Big( r(u)^2\int_{\R^N}|\nabla u|^2\, dx-
		2^*\int_{\R^N}G(u)\, dx\Big)\\
		&\hspace{-10mm}+ r(u)^{-N}\Big(r(u)^2\int_{\R^N}\langle \nabla u,  \nabla v\rangle\,dx-\int_{\R^N}g(u)v\, dx\Big)\\
		&\hspace{-15mm}=\frac{2-N}{2}r(u)^{-1} r'(u)(v) M(m_{\cP}(u)) + J'(m_\cP(u))(v(r(u)\cdot)\\
		&\hspace{-15mm}=J'(m_{\cP}(u))(v(r(u)\cdot).
		\end{aligned}
	$$
(iii) Let us introduce the following auxiliary functions $g_1(s)=\max\{g(s)+ms,0\}$ and $g_2(s)=g_1(s)-g(s)$ for $s\geq 0$ and $g_i(s)=-g_i(-s)$ for $s<0$. Then $g_1(s),g_2(s)\geq 0$ for $s\geq 0$,
\begin{eqnarray}\label{eq:NewCond1}
\lim_{s\to 0} g_1(s)/s&=&\lim_{s\to \infty}g_1(s)/s^{2^*-1}=0\\
g_2(s)&\geq& m s \quad \hbox{ for }s\geq 0,\label{eq:NewCond2}
\end{eqnarray}
and let
$$G_i(s)=\int_0^s g_i(t)\, dt\quad \hbox{ for }i=1,2.$$ 
The condition \eqref{eq:NewCond1} will be important e.g. to apply Theorem \ref{ThGerard} with $\Psi(s)=G_1(s)$, whereas  \eqref{eq:NewCond2} is used below to estimate the $L^2$-norm. Namely,
suppose that for some $(u_n)\subset\cU$
$$J(m_\cU(u_n))=\Big(\frac12-\frac{1}{2^*}\Big)\int_{\R^N}|\nabla m_\cU(u_n)|^2\, dx$$
is bounded.  Then we obtain that $m_\cU(u_n)$ is bounded in $L^{2^*}(\R^N)$ and by \eqref{eq:NewCond1}, $\int_{\R^N}G_1(m_\cU(u_n))\, dx$ is bounded as well. By \eqref{eq:NewCond2} and since $m_\cU(u_n)\in\cM$, we infer that $m_\cU(u_n)$ is bounded in $H^1(\R^N)$. Thus $J$ is coercive on $\cM$. Observe that 
for some constants $0<C_1<C_2$ one has
\begin{eqnarray*}
|m_\cU(u_n)|_{2^*}^2+|m_\cU(u_n)|_{2}^2&\leq& C_1 \int_{\R^N}\big(|\nabla m_\cU(u_n)|^2+2^*G_2(m_\cU(u_n))\big)\, dx
=C_1 2^*\int_{\R^N}G_1(m_\cU(u_n))\, dx\\
& \leq& |m_\cU(u_n)|_{2}^2+C_2 |m_\cU(u_n)|_{2^*}^{2^*}
\end{eqnarray*}
and 
we conclude that $|m_\cU(u_n)|_{2^*}\geq C_2^{-1/(2^*-2)}>0$. Hence $c=\inf_{\cM} J>0$.\\
(iv) Note that if $u_n\to u\in\partial \cU$ and $u_n\in \cU$, then $r(u_n)\to 0$ and 
$$\|m_\cU(u_n)\|^2=r(u_n)^{2-N}+r(u_n)^{-N}|u_n|_2^2\to \infty$$
as $n\to\infty$. Hence by the coercivity,
$J(m_\cU(u_n))\to\infty$ as $n\to\infty$.
\end{proof}

Now observe that we may consider the group of translations $G=\R^N$ acting on $X=H^1(\R^N)$, i.e.  
$$(yu)(x)=u(x+y)$$ for $y\in\R^N$, $u\in X$, $x\in\R^N$,
and in view of Proposition \ref{prop:defMPSU} conditions (A1)--(A3) are satisfied.\\
\indent In the similar way we may consider the following subgroup of translations $G=\{0\}\times\{0\}\times\R^{N-2M}$ acting on $X=X_\tau\cap H^1_{\cO_1}(\R^N)$ and conditions (A1)--(A3) are satisfied provided that instead of  $\cM$, $\cS$, $Y=X=H^1(\R^N)$, $\cU$, $m_\cP$ and $m$, we consider $\cM\cap X$, $\cS\cap X$, $Y=X=X_\tau\cap H^1_{\cO_1}(\R^N)$, $\cU\cap X$, $m|_{\cP\cap X}:\cP\cap X\to \cM\cap X$ and $m|_{\cU\cap X}:\cU\cap X\to \cM\cap X$ respectively.\\
\indent Finally, in case of $X=X_\tau\cap H^1_{\cO_2}(\R^N)$ we consider the trivial group $G=\{(0,0,0)\}$ acting on  $X$.

\begin{Rem}\label{remTau}
We show how to easy construct functions in $X_\tau\cap H^1_{\cO_2}(\R^N)$. Let 
$u\in H^1_0(B(0,R))\cap L^{\infty}(B(0,R))$ be $\cO(N)$-invariant (radial) function, $R>1$ and take any odd and smooth function $\vp:\R\to [0,1]$ such that $\vp(x)=1$ for $x\geq 1$ and $\vp(x)=-1$ for $x\leq -1$. Note that, defining
\begin{equation*}
\tu(x_1,x_2,x_3):=u\big(\sqrt{|x_1|^2+|x_2|^2+|x_3|^2}\big)\vp(|x_1|-|x_2|)\;\hbox{ for }x_1,x_2\in\R^M,x_3\in\R^{N-2M},
\end{equation*}
we get $\tu\in X_\tau\cap H^1_{\cO_2}(\R^N)$. Take $A:=\mathrm{ess \sup }\;|u|$ and $B:=\max_{s\in [0,A]}|G(s)|$. Let us denote $r=|x|$ and $r_i=|x_i|$ for $i=1,2,3$. Observe that
\begin{eqnarray*}
\int_{\R^N} G(\tu)\, d x&=&\int_{0}^\infty\int_{0}^\infty\int_0^\infty  G(\tu)r_1^{m-1}r_2^{m-1}r_3^{N-2M-1}\, d r_1dr_2 dr_3\\
&=&2\int_{0}^R\int_{0}^R\int_{r_2}^{r_2+R}G(\tu)r_1^{m-1}r_2^{m-1}r_3^{N-2M-1}\, d r_1dr_2 dr_3
\end{eqnarray*}
\begin{eqnarray*}
&=&2\int_{0}^R\int_{0}^R\int_{r_2}^{r_2+R}G(u(r))r_1^{m-1}r_2^{m-1}r_3^{N-2M-1}\, d r_1dr_2 dr_3\\
&&-2\int_{0}^R\int_{0}^R\int_{r_2}^{r_2+1}G(u(r))r_1^{m-1}r_2^{m-1}r_3^{N-2M-1}\, d r_1dr_2 dr_3\\
&&+2\int_{0}^R\int_{0}^R\int_{r_2}^{r_2+1}G(u(r)\vp(r_1-r_2))r_1^{m-1}r_2^{m-1}r_3^{N-2M-1}\, d r_1dr_2 dr_3\\
&\geq& \int_{\R^N} G(u)\, dx-c_1B\Big(\sum_{i=N-m}^{N-1}R^i\Big)
\end{eqnarray*}
for some constant $c_1>0$ dependent only on $N$. In \cite{BerLions}[page 325], for any $R>0$ one can find a radial function $u\in H^1_0(B(0,R))\cap L^{\infty}(B(0,R))$ such that $\int_{\R^N}G(u)\, dx\geq c_2 R^N-c_3 R^{N-1}$ for some constants $c_2,c_3>0$. Therefore we get $\int_{\R^N}G(\tu)\, dx>0$ for sufficiently large $R$, hence $\cP\cap X_\tau\cap H^1_{\cO_2}(\R^N)\neq\emptyset$ and $\cM\cap X_\tau\cap H^1_{\cO_2}(\R^N)\neq\emptyset$.
\end{Rem}
 
\subsection{$\theta$-analysis of Palais-Smale sequences}

Below we explain the role of $\theta$ in the analysis of Palais-Smale sequences of $J\circ m_\cU$.

\begin{Lem}\label{lem:theta}
Suppose that $(u_n)\subset \cU$ is a $(PS)_\beta$-sequence of $J\circ m_\cU$ such that 
	$$m_\cU(u_n)(\cdot +y_n)\weakto\tu\neq 0\hbox{ in } H^1(\R^N)$$ for some sequence $(y_n)\subset \R^N$ and $\tu\in H^1(\R^N)$. Then $\tu$ solves
	\begin{equation}\label{eq:thetaEQ}
	-\theta \Delta u = g(u),\hbox{ where }\theta:=\psi(\tu)^{-1}\int_{\R^N}g(\tu)\tu\,dx,
	\end{equation}
	and  passing to a subsequence 
	\begin{equation}\label{eq:theta2}
	\theta=\lim_{n\to\infty}\psi(m_\cU(u_n))^{-1}\int_{\R^N} g(m_\cU(u_n))m_\cU(u_n)\,dx.
	\end{equation}
	Moreover $\theta\neq 0$ and
	\begin{equation}\label{eq:theta}
	\theta=2^*\psi(\tu)^{-1}\int_{\R^N}G(\tu)\,dx.
	\end{equation}
If $\theta>0$, then $m_{\cP}(\tu)\in\cM$ is a critical point of $J$. If $\theta\geq 1$, then $J(m_{\cP}(\tu))\leq \beta$.
	\end{Lem}
\begin{proof}
For $v\in X$ we set $v_n(x)=v(r(u_n)^{-1}x-y_n)$ and observe that passing to a subsequence $m_\cU(u_n)(x +y_n)\to\tu(x)$  for a.e. $x\in\R^N$ and by Vitali's convergence theorem
\begin{eqnarray}\nonumber
(J\circ m_\cU)'(u_n)(v_n)
&=& \int_{\R^N}\langle \nabla m_\cU(u_n)(\cdot+y_n),\nabla v\rangle\, dx -
\int_{\R^N}g(m_\cU(u_n)(\cdot+y_n))v\,dx\\\label{eq:conv22q}
&\to& \int_{\R^N}\langle \nabla \tu,\nabla v\rangle\, dx -\int_{\R^N}g(\tu)v\,dx.
\end{eqnarray}
We find the following decomposition
$$v_n=\Big(\int_{\R^N}\langle \nabla u_n,\nabla v_n\rangle\,dx\Big)u_n+\tv_n$$
with $\tv_n\in T_{u_n}\cS$.  In view of Proposition \ref{prop:defMPSU} (iii) we get that $r(u_n)$ is bounded from above, bounded away from $0$ and  passing to a subsequence $r(u_n)\to r_0>0$. Note that $(v_n)$ is bounded, hence $(\tv_n)$ is bounded and $(J\circ m_\cU)'(u_n)(\tv_n)\to 0$. Moreover
\begin{eqnarray*}
\int_{\R^N}\langle \nabla u_n,\nabla v_n\rangle\,dx&=&
r(u_n)^{N-2}
\int_{\R^N}\langle \nabla m_\cU(u_n), \nabla v_n(r(u_n)\cdot)\rangle\,dx\\
&=&r(u_n)^{N-2}\int_{\R^N}\langle \nabla m_\cU(u_n)(\cdot+y_n), \nabla v\rangle\,dx\\
&\to& r_0^{N-2} \int_{\R^N}\langle \nabla \tu,\nabla v\rangle\,dx=0
\end{eqnarray*}
provided that $\int_{\R^N}\langle\nabla \tu,\nabla v \rangle\, dx=0$.
Hence
$$(J\circ m_\cU)'(u_n)(v_n)=\Big(\int_{\R^N}\langle \nabla u_n,\nabla v_n\rangle\,dx\Big)(J\circ m_\cU)'(u_n)(u_n)+(J\circ m_\cU)'(u_n)(\tv_n)\to 0$$
and by \eqref{eq:conv22q} we obtain
$$\int_{\R^N}\langle \nabla \tu,\nabla v\rangle\, dx -\int_{\R^N}g(\tu)v\,dx=0$$
for any $v$ such that $\int_{\R^N}\langle\nabla \tu,\nabla v \rangle\, dx=0$. We define
$\xi:H^1(\R^N)\to \R$ by the following formula
\begin{eqnarray*}
\xi(v)&=&\int_{\R^N}\langle \nabla \tu,\nabla v\rangle\, dx -\int_{\R^N}g(\tu)v\,dx\\
&&-\Big(\int_{\R^N}|\nabla \tu|^2\, dx -\int_{\R^N}g(\tu)\tu\,dx\Big)
\psi(\tu)^{-1}\int_{\R^N}\langle\nabla \tu,\nabla v\rangle\, dx.
\end{eqnarray*}
Observe that any $v\in H^1(\R^N)$ has the following decomposition
$$v=\Big(\int_{\R^N}\langle\nabla \tu,\nabla v\rangle\, dx \Big)\tu+\tv$$
such that $\int_{\R^N}\langle \nabla \tu,\nabla \tv\rangle\,dx=0$. Note that $\xi(\tu)=0$ and
\begin{eqnarray*}
\xi(v)&=&\Big(\int_{\R^N}\langle\nabla \tu,\nabla v\rangle\, dx \Big)\xi(\tu)+\xi(\tv)\\
&=& \xi(\tv)=0
\end{eqnarray*}
for any $v\in H^1(\R^N)$.
Then
\begin{eqnarray}\label{eq:xieq}
0=\xi(v)&=& \int_{\R^N}\Big(1-\Big(\int_{\R^N}|\nabla \tu|^2\, dx -\int_{\R^N}g(\tu)\tu\,dx\Big)
\psi(\tu)^{-1}\Big)\langle\nabla  \tu,  \nabla v\rangle\,dx\\
&&-\int_{\R^N}g(\tu)v\, dx\nonumber
\end{eqnarray}
and $\tu$ is a weak solution to the problem \eqref{eq:thetaEQ}
with 
$$\theta=1-\Big(\int_{\R^N}|\nabla \tu|^2\, dx -\int_{\R^N}g(\tu)\tu\,dx\Big)\psi(\tu)^{-1}
=\psi(\tu)^{-1}\int_{\R^N}g(\tu)\tu\,dx.$$
Now we show \eqref{eq:theta2}. Let us define a map $\eta:\cP\to (H^1(\R^N))^*$ by the following formula
\begin{eqnarray*}
	\eta(u)(v)&=& (J\circ m_\cP)'(u)(v)-(J\circ m_\cP)'(u)(u)\int_{\R^N}\langle \nabla u,\nabla v\rangle\,dx
\end{eqnarray*}
for $u\in\cP$ and $v\in H^1(\R^N)$.
Observe that any $v\in H^1(\R^N)$ has the unique decomposition
$$v=\Big(\int_{\R^N}\langle\nabla u_n,\nabla v\rangle\, dx \Big)u_n+\tv_n$$
such that $\tv_n\in T_{u_n}\cS$. Note that
\begin{eqnarray*}
	\eta(u_n)(v)&=&\Big(\int_{\R^N}\langle\nabla u_n,\nabla v\rangle\, dx \Big)\eta(u_n)(u_n)+\eta(u_n)(\tv_n)\\
	&=& \eta(u_n)(\tv_n)=(J\circ m_\cU)'(u_n)(\tv_n).
\end{eqnarray*}
Since $(u_n)$ is a $(PS)_\beta$-sequence of $J\circ m_\cU$,  we obtain $\eta(u_n)\to 0$ in $(H^1(\R^N))^*$. On the other hand, in view of Proposition \ref{prop:defMPSU} (ii)
\begin{eqnarray*}
	\eta(u_n)(v(r(u_n)^{-1}x-y_n))
	&=& \int_{\R^N}\big(1-r(u_n)^{N-2}(J\circ m_\cP)'(u_n)(u_n)\big)
	\langle\nabla  m_\cU(u_n)(\cdot+y_n),  \nabla v\rangle\,dx\\
	&&-\int_{\R^N}g(m_\cU(u_n)(\cdot+y_n))v\, dx\\
	&=& \int_{\R^N}\theta_n
	\langle\nabla  m_\cU(u_n)(\cdot+y_n),  \nabla v\rangle\,dx-\int_{\R^N}g(m_\cU(u_n)(\cdot+y_n))v\, dx,
\end{eqnarray*}
where
$$\theta_n=r(u_n)^{N-2}\int_{\R^N}g(m_\cU(u_n))m_\cU(u_n)\,dx=\psi(m_\cU(u_n))^{-1}\int_{\R^N}g(m_\cU(u_n))m_\cU(u_n)\,dx.$$
Passing to a subsequence $\theta_n\to\tilde{\theta}$ and
\begin{eqnarray*}
	0=\lim_{n\to\infty}\eta(u_n)(v(r(u_n)^{-1}x-y_n))
	&=& \int_{\R^N}\tilde{\theta}
	\langle\nabla  \tu,  \nabla v\rangle\,dx-\int_{\R^N}g(\tu)v\, dx
\end{eqnarray*}
for any $v\in H^1(\R^N)$. Taking into account \eqref{eq:xieq} we obtain that $\tilde{\theta}=\theta$ and
\eqref{eq:theta2} is satisfied.
Now we show that $\theta\neq 0$.  Suppose that $\theta=0$, hence $g(\tu(x))=0$ for a.e. $x\in\R^N$. 
 Take 
$\Sigma:=\{x\in\R^N: g(\tu(x))=0\}$ 
and clearly $\R^N\setminus \Sigma$ has measure zero and let
$\Om:=\{x\in \Sigma: \tu(x)\neq 0\}$.
Suppose that
$\eps:=\mathrm{essinf}_{x\in\Om}|\tu(x)|>0$. Since $\tu\in L^2(\R^N)\setminus\{0\}$, we infer that $\Om$ has finite positive measure 
and note that
$$\int_{\R^N}|\tu(x+h)-\tu(x)|^2\,dx\geq \eps \int_{\R^N}|\chi_\Om(x+h)-\chi_\Om(x)|^2\,dx$$
holds  for any $h\in\R^N$,
where $\chi_\Om$ is the characteristic function of $\Om$. In view of \cite{Ziemer}[Theorem 2.1.6] we infer that $\chi_\Om\in H^1(\R^N)$ and we get the contradiction with the assumption $\eps>0$. Therefore, $\eps=\mathrm{essinf}_{x\in\Om}|\tu(x)|=0$, and we find
 a sequence $(x_n)\subset\R^N$ such that $\tu(x_n)\to 0$, $\tu(x_n)\neq 0$ and $g(\tu(x_n))=0$. Thus, by \eqref{eq:NewCond1} and \eqref{eq:NewCond2} we obtain the next contradiction
$$0=\lim_{n\to\infty}\frac{g(\tu(x_n))}{\tu(x_n)}\leq -m<0.$$
Therefore $\theta\neq 0$ and
by the elliptic regularity we infer that
$\tu\in W^{2,q}_{loc}(\R^N)$ for any $q<\infty$. In view of the Pohozaev identity
$$\theta \int_{\R^N}|\nabla \tu|^2\, dx=2^*\int_{\R^N}G(\tu)\,dx,$$
hence \eqref{eq:theta} holds.
Now suppose that $\theta>0$. Then 
$$r(\tu)=\Big(\frac{2^*\int_{\R^N}G(\tu)\, dx}{\psi(\tu)}\Big)^{1/2}=\Big(\frac{\int_{\R^N}g(\tu)\tu\, dx}{\psi(\tu)}\Big)^{1/2}=\theta^{1/2}.$$
Observe that for $v\in X$ and $v_r=v(r(\tu)^{-1}\cdot)$ one has
\begin{eqnarray*}
J'(m_{\cP}(\tu))(v)&=&r(\tu)^{2-N}\int_{\R^N}\langle \nabla \tu,\nabla v_r\rangle\, dx -r(\tu)^{-N}\int_{\R^N}g(\tu)v_r\,dx\\
&=&r(\tu)^{-N}\Big(\int_{\R^N}\langle \nabla \theta \tu,\nabla v_r\rangle\, dx -\int_{\R^N}g(\tu)v_r\,dx\Big)=0,
\end{eqnarray*}
which finally shows that $m_{\cP}(\tu)$ is a critical point of $J$. If $\theta\geq 1$, then
\begin{eqnarray*}
\beta=\lim_{n\to\infty}J(m_\cU(u_n))&\geq& \Big(\frac{1}{2}-\frac{1}{2^*}\Big)\int_{\R^N}|\nabla \tu|^2\, dx
\geq r(\tu)^{2-N}\Big(\frac{1}{2}-\frac{1}{2^*}\Big)\int_{\R^N}|\nabla \tu|^2\, dx\\
&=&J(m_{\cP}(\tu)).
\end{eqnarray*}
\end{proof}

The main difficulty in the analysis of Palais-Smale sequences of $J\circ m_\cU$ is to find proper translations $(y_n)\subset \R^N$ such that $\theta>0$ in Lemma \ref{lem:theta}. In order to check $(M)_\beta\; (i)$ condition one needs to ensure that even $\theta\geq 1$. This can be performed with the help of 
the following result providing decompositions for Palais-Smale sequences of $J\circ m_\cU$, which is based on the profile decomposition Theorem \ref{ThGerard}. Observe that in the usual variational approach e.g. due to Struwe \cite{StruweSplitting} or  Coti Zelati and Rabinowitz \cite{CotiZelatiRab}, such decompositions of Palais-Smale sequences are finite. In our case, however, a finite procedure cannot be performed in general, since we do not know whether a weak limit point of a Palais-Smale sequence of $J\circ m_\cU$ is a critical point. Therefore we need to employ the profile decompositions from Theorem \ref{ThGerard} and Corollary \ref{CorGerard}. 

\begin{Prop}\label{prop:PSanaysis}
Let $(u_n)\subset \cU$ be a Palais-Smale sequence of $J\circ m_\cU$ at level $\beta=c$.
Then there is $K\in \N\cup \{\infty\}$ and there are sequences
$(\tu_i)_{i=0}^K\subset H^1(\R^N)$, $(\theta_i)_{i=0}^K\subset \R$, for any $n\geq 1$, $(y_n^i)_{i=0}^K\subset \R^N$ is such that $y_n^0=0$,
$|y_n^i-y_n^j|\rightarrow \infty$ as $n\to\infty$ for $i\neq j$, and passing to a subsequence, the following conditions hold:
\begin{eqnarray}\label{EqSplit1}
&& m_\cU(u_n)(\cdot+y_n^i)\weakto \tu_i\; \hbox{ in } H^1(\R^N)\text{ as }n\to\infty\text{ for } 0\leq i <K+1,\\
\label{EqSplit2}
&& 
\tu_i\text{ solves }\eqref{eq:thetaEQ}\text{ with }\theta_i\text{ for } 0\leq i <K+1,\tu_i\neq 0\text{ and }\eqref{eq:theta}\text{ holds  for } 1\leq i <K+1,\\\nonumber
&&\text{if }\tu_0\neq 0,\text{ then }\theta_0\neq 0\text{ and satisfies }\eqref{eq:theta},\\
\label{EqSplit3}
&& \lim_{n\to\infty}\int_{\R^N}G_1(m_\cU(u_n))\, dx= \sum_{i=0}^K
\int_{\R^N}G_1(\tu_i)\, dx,
\\\label{EqSplit4}
&& \lim_{n\to\infty}\int_{\R^N}G_2(m_\cU(u_n))\, dx\geq \sum_{i=0}^K
\int_{\R^N}G_2(\tu_i)\, dx,\\\label{EqSplit5}
&&\lim_{n\to\infty}\psi(m_\cU(u_n))\geq \sum_{i=0}^K \psi(\tu_i).
\end{eqnarray}
\end{Prop}
\begin{proof}
Since $J$ is coercive on $\cM$, we know that $m_\cU(u_n)$ is bounded and passing to a subsequence we may assume that $\lim_{n\to\infty}\int_{\R^N}G_1(m_\cU(u_n))\, dx$, $\lim_{n\to\infty}\int_{\R^N}G_2(m_\cU(u_n))\, dx$ exist. In view of Theorem \ref{ThGerard} we obtain  sequences
$(\tu_i)_{i=0}^\infty\subset H^1(\R^N)$ and $(y_n^i)_{i=0}^\infty\subset \R^N$ for  $n\geq 1$, such that \eqref{EqSplit2a}--\eqref{EqSplit3a} are satisfied. If $(\tu_i)_{i=1}^\infty$ contains exactly $K$ nontrivial functions, then we may assume that $\tu_i\neq 0$ for $i=1,...,K$. Otherwise we set $K=\infty$. In view of Lemma \ref{lem:theta}, $\tu_i$ solves \eqref{eq:thetaEQ} with $\theta_i$ for $1\leq i< K+1$. If $\tu_0=0$, then $\theta_0=0$, and if $\tu_0\neq 0$, then $\theta_0$ is given by \eqref{eq:thetaEQ}, so that \eqref{EqSplit1}-\eqref{EqSplit2} hold. Since \eqref{eq:NewCond1} holds, then \eqref{EqSplit3} follows from \eqref{EqSplit3a} and \eqref{EqSplit4a}. Moreover \eqref{EqSplit4} and \eqref{EqSplit5} follow from \eqref{EqSplit3a} and \eqref{EqSplit2a} respectively.
\end{proof}

\begin{Cor}\label{cor:PSanalysis}
	Suppose that $X=H^{1}_{\cO_1}(\R^N)\cap X_\tau$  and $2\leq M< N/2$.
	Let $(u_n)\subset \cU\cap X$ be a Palais-Smale sequence of $J|_X\circ m_\cU|_{\cU\cap X}$ at level $\beta=\inf_{\cM\cap X}J$.
	Then there is $K\in \N\cup \{\infty\}$ and there are sequences
	$(\tu_i)_{i=0}^K\subset X$, $(\theta_i)_{i=0}^K\subset \R$, $(y_n^i)_{i=0}^K\subset  \{0\}\times\{0\}\times\R^{N-2M}$ for any $n\geq 1$ such that $y_n^0=0$,
	$|y_n^i-y_n^j|\rightarrow \infty$ as $n\to\infty$ for $i\neq j$, and passing to a subsequence, \eqref{EqSplit1}--\eqref{EqSplit5} are satisfied. 
\end{Cor}
\begin{proof} We argue as in proof of Proposition \ref{prop:PSanaysis}, but instead of Theorem \ref{ThGerard} we use Corollary \ref{CorGerard}.
Arguing as in Lemma \ref{lem:theta} we obtain that $\tu_i$ solves \eqref{eq:thetaEQ} with $\theta_i$ in $X$ for $0\leq i <K+1$, i.e.
$$\theta_i\int_{\R^N}\langle\nabla \tu_i, \nabla v\rangle\,dx=\int_{\R^N} g(\tu_i)v\,dx\quad\hbox{for every }v\in X,$$
and $\tu_i\neq 0$ for $1\leq i <K+1$. By the Palais principle of symmetric criticality \cite{Palais}, $\tu_i$ solves \eqref{eq:thetaEQ} with $\theta_i$. As in Lemma \ref{lem:theta} we show that $\theta_i\neq 0$ and by the Pohozaev identity \eqref{eq:theta} holds for $\tu_i$ and $\theta_i$ for $1\leq i <K+1$. If $\tu_0=0$, then $\theta_0=0$, otherwise $\theta_0$ is given by \eqref{eq:thetaEQ} and that the Pohozaev identity satisfies also \eqref{eq:theta}.
\end{proof}

\begin{Rem}\label{Rem:theta}
An important consequence of Proposition \ref{prop:PSanaysis} and Corollary \ref{cor:PSanalysis} is the existence of a sequence of translations $(y^i_n)$ such that $\theta_i\geq 1$ for some $i\geq 0$. Indeed, in view of \eqref{EqSplit2} we get
$$\theta_i\psi(\tu_i)=2^*\Big(\int_{\R^N}G_1(\tu_i)\, dx-\int_{\R^N}G_2(\tu_i)\, dx\Big)$$
for $0\leq i < K+1$. Then by \eqref{EqSplit3}--\eqref{EqSplit5} we obtain
\begin{eqnarray*}
\sum_{i=0}^K\theta_i\psi(\tu_i)&=&2^*\Big(\sum_{i=0}^K\int_{\R^N}G_1(\tu_i)\, dx-\sum_{i=0}^K\int_{\R^N}G_2(\tu_i)\, dx\Big)\\
&\geq& 2^*\Big(\lim_{n\to\infty}\int_{\R^N}G_1(m_\cU(u_n))\, dx - \lim_{n\to\infty}\int_{\R^N}G_2(m_\cU(u_n))\, dx\Big)\\
&=& \lim_{n\to\infty}\psi(m_\cU(u_n))\geq \sum_{i=0}^K \psi(\tu_i).
\end{eqnarray*}
Therefore there is $\theta_i\geq 1$ for some $0\leq i<K+1$.
\end{Rem}

\subsection{Proof of Theorems \ref{ThMain1}, \ref{ThMain2} and \ref{ThMain3}}
\begin{Lem}\label{lem:Mcond}
$J\circ m_\cU$ satisfies $(M_\beta)\; (i)$ for $\beta=c$. 
\end{Lem}
\begin{proof}
	Let $(u_n)\subset \cU$ be a $(PS)_\beta$-sequence of $J\circ m_\cU$.  Since $J$ is coercive on $\cM$, $(m_\cU(u_n))$ is bounded and  in view of Proposition \ref{prop:PSanaysis} and Remark \ref{Rem:theta} we find a sequence $(y_n)\subset \R^N$ such that $m_\cU(u_n)(\cdot+y_n)\weakto \tu$ in $H^1(\R^N)$ for some $\tu\neq 0$ and $\theta\geq 1$ given by \eqref{eq:theta}. Observe that $\tu\in\cP$ and by Lemma \ref{lem:theta} we conclude.
\end{proof}

Now, let us consider $\cO_1$-invariant functions.

\begin{Lem}\label{lem:Mcond2}
Suppose that $X=H^{1}_{\cO_1}(\R^N)\cap X_\tau$ and $2\leq M< N/2$. Then $J|_X\circ m_\cU|_{\U\cap X}$ satisfies  $(M_\beta)\; (i)$ for $\beta=\inf_{\cM\cap X}J$.
\end{Lem}
\begin{proof}
Let $(u_n)\subset \cU\cap X$ be a $(PS)_\beta$-sequence of $J|_X\circ m_\cU|_{\U\cap X}$.  Similarly as in  proof of Lemma \ref{lem:Mcond}, in view of Corollary \ref{cor:PSanalysis} and and Remark \ref{Rem:theta} we find a sequence $(y_n)\subset \{0\}\times\{0\}\times\R^{N-2M}$ such that $m_\cU(u_n)(\cdot+y_n)\weakto \tu$ in $X$ for some $\tu\neq 0$ and $\theta\geq 1$ given by \eqref{eq:theta}. Observe that $\tu\in\cP\cap X$ and as in Lemma \ref{lem:theta}, $J(m_\cP(\tu))\leq \beta$.

\end{proof}

More can be said for $\cO_2$-invariant functions.
\begin{Lem}\label{lem:Mcond3}
	Suppose that $X=H^{1}_{\cO_2}(\R^N)\cap X_\tau$. If $(u_n)\subset\cU\cap X$ is a $(PS)_\beta$-sequence of $(J|_X\circ m_\cU|_{\U\cap X})$, then passing to a subsequence $u_n\to u_0$ for some $u_0\in\cU\cap X$ such that $J|'_X(m_\cU(u_0))=0$.
\end{Lem}
\begin{proof}
Let $(u_n)\subset \cU\cap X$ be a sequence such that $(J|_X\circ m_\cU|_{\U\cap X})'(u_n)\to 0$ and $(J|_X\circ m_\cU|_{\U\cap X})(u_n)\to\beta$.  Since $J$ is coercive on $\cM$, $(m_\cU(u_n))$ is bounded and  in view of Corollary \ref{cor:LionsO2} we find $\tu\in X$ such that 
\begin{equation}\label{eq:G_1conv}
\int_{\R^N}G_1(m_\cU(u_n))\, dx\to \int_{\R^N}G_1(\tu)\, dx
\end{equation}
as $n\to\infty$. If $\tu=0$, then by \eqref{eq:NewCond2}
$$\min\big\{1,\frac{m}{2}\big\}\|m_\cU(u_n)\|^2\leq\int_{\R^N}|\nabla m_\cU(u_n)|^2\,dx+2^*\int_{\R^N}G_2(m_\cU(u_n))\,dx=2^*\int_{\R^N}G_1(m_\cU(u_n))\,dx\to 0,$$
which contradicts the fact that $\inf_{\cM\cap X}J>0$. Therefore $\tu\neq 0$.
Now, observe that applying Corollary \ref{cor:LionsO2}  with \eqref{EqPsiLast} for $\Psi(s)=G_1(s)$ and passing to subsequence we obtain 
\begin{eqnarray*}
	\lim_{n\to\infty}\int_{\R^N} g_1(m_\cU(u_n))m_\cU(u_n)\,dx&=&\int_{\R^N} g_1(\tu)\tu\,dx,\\
	\lim_{n\to\infty}\int_{\R^N} g_2(m_\cU(u_n))m_\cU(u_n)\,dx&\geq&\int_{\R^N} g_2(\tu)\tu\,dx.
\end{eqnarray*}
Similarly as in Lemma \ref{lem:theta} we infer that \eqref{eq:theta2}, \eqref{eq:theta} hold and 
\begin{eqnarray*}
	\psi(\tu)&\leq&\lim_{n\to\infty}\psi(m_\cU(u_n))=\lim_{n\to\infty}\int_{\R^N}g_1(m_\cU(u_n))m_\cU(u_n)\,dx-\lim_{n\to\infty}\int_{\R^N}g_2(m_\cU(u_n))m_\cU(u_n)\,dx\\
	&\leq&\int_{\R^N}g_1(\tu)\tu\,dx-\int_{\R^N}g_2(\tu)\tu\,dx= \theta \psi(\tu).
\end{eqnarray*}	
Hence $\theta\geq 1$ and $m_{\cP}(u)\in \cM\cap X$ is a critical point of $J|_X\circ m_\cU|_{\U\cap X}$.
 In view of \eqref{eq:theta2} we get
 $$\theta=\lim_{n\to\infty}\psi(m_\cU(u_n))^{-1}\int_{\R^N} g(m_\cU(u_n))m_\cU(u_n)\,dx\leq 
 \psi(\tu)^{-1}\int_{\R^N} g(\tu)\tu\,dx=\theta,$$
hence 
 $$\lim_{n\to\infty}\psi(m_\cU(u_n))=\psi(\tu)$$
 and 
 $$\lim_{n\to\infty}\int_{\R^N} g_2(m_\cU(u_n))m_\cU(u_n)\,dx=\int_{\R^N} g_2(\tu)\tu\,dx.$$
 Note that $g_2(s)=ms+g_3(s)$, where $g_3(s):=\max\{0,-g(s)-ms\}\geq 0$ for $s\geq 0$ and $g_3(s):=-\max\{0,-g(s)-ms\}$ for $s\leq 0$. Then $g_3(s)s\geq 0$ for $s\in\R$ and we easy infer that
 $$\lim_{n\to\infty}\int_{\R^N}  m|m_\cU(u_n)|^2\,dx=\int_{\R^N} m|\tu|^2\,dx.$$
 Therefore $m_\cU(u_n)\to \tu$ in $H^1(\R^N)$, $\tu\in\cM\cap X$ and $u_n\to u_0:=m^{-1}(\tu)$ in $\cU\cap X$. Since $\theta=1$, $J|'_X(m_\cU(u_0))=J|'_X(\tu)=0$.
\end{proof}

\begin{altproof}{Theorem \ref{ThMain1}}
Since $J$  satisfies (A1)--(A3) and $(M_\beta)\;(i)$, proof follows from Theorem \ref{Th:CrticMulti} (a).
\end{altproof}

\begin{altproof}{Theorem \ref{ThMain2}}
If $X=H^{1}_{\cO_1}(\R^N)\cap X_\tau$ and $2\leq m< N/2$, then $J|_X\circ m_\cU|_{\U\cap X}$ satisfies
 (A1)--(A3) and $(M_\beta)\;(i)$. Then, in view of Theorem \ref{Th:CrticMulti} (a) there is a critical point $u\in \cM\cap X$ of $J|_X$ such that 
 $$J(u)=\inf_{\cM\cap X}J.$$
 In view of the Palais principle of symmetric criticality \cite{Palais}, $u$ solves \eqref{eq}.
Let
\begin{eqnarray*}
\Omega_1&:=&\{x\in\R^N: |x_1|>|x_2|\},\\
\Omega_2&:=&\{x\in\R^N: |x_1|<|x_2|\}.\\
\end{eqnarray*}
Since $u\in X_\tau\cap H^1_{\cO_1}(\R^N)$, we get $\chi_{\Omega_1} u\in H^1_0(\Omega_1)\subset H^1(\R^N)$ and 
 $\chi_{\Omega_2} u\in H^1_0(\Omega_2)\subset H^1(\R^N)$. Moreover $\chi_{\Omega_1} u\in \cM$ and
$$J(u)=J(\chi_{\Omega_1} u)+J(\chi_{\Omega_2} u)=2J(\chi_{\Omega_1} u)\geq 2 \inf_{\cM} J,$$
which completes the proof of \eqref{eq:thmain2}. The remaining case $2\leq m=N/2$ is contained in Theorem \ref{ThMain3}.
\end{altproof}

\begin{altproof}{Theorem \ref{ThMain3}}
If $X=H^{1}_{\cO_2}(\R^N)\cap X_\tau$, then $J|_X\circ m_\cU|_{\U\cap X}$ satisfies
(A1)--(A3), and note that Lemma \ref{lem:Mcond3} holds.  In view of \cite{BerLionsII}[Theorem 10], for any $k\geq 1$ we find an odd continuous map $\tau:S^{k-1}\to H_0^1(B(0,R))\cap L^{\infty}(B(0,R))$ such that $\tau(\sigma)$ is a radial function and $\tau(\sigma)\neq 0$ for all $\sigma\in S^{k-1}$. Moreover 
$$\int_{B(0,R)}G(\tau(\sigma))\,dx\geq c_2R^N-c_3R^{N-1}$$
for any $\sigma\in S^{k-1}$ and some constants $c_2,c_3>0$. As in Remark \ref{remTau} we define a map $\tilde{\tau}:S^{k-1}\to H_0^1(B(0,R))\cap L^{\infty}(B(0,R))$ such that $\tilde{\tau}(\sigma)(x_1,x_2,x_3)=\tau(\sigma)(x_1,x_2,x_3)\vp(|x_1|-|x_2|)$. Observe that $\tilde{\tau}(\sigma)\in X$ and
$\int_{B(0,R)}G(\tilde{\tau}(\sigma))\,dx>0$ for $\sigma\in S^{k-1}$ and sufficiently large $R$. Therefore (S) is satisfied and
 proof follows from Theorem \ref{Th:CrticMulti} (c) and from the Palais principle of symmetric criticality \cite{Palais}.
\end{altproof}

{\bf Acknowledgements.}
I would like to thank Louis Jeanjean and Sheng-Sen Lu, who observed that one has to ensure that the Lusternik-Schnirelmann values are finite in proof of Theorem 2.2. Inspired by this work, very recently they also recovered the results of this paper in  \cite{LuJeanjean} by means of the symmetric mountain pass setting and the monotonicity trick.  Moreover I am  grateful to  Rupert Frank for pointing out the reference \cite{Nawa} and  I am also indebted to the referees for many valuable comments.\\
\indent The author was partially supported by the National Science Centre, Poland (Grant No. 2014/15/D/ST1/03638).

\end{document}